\input ./style/arxiv-general.cfg
\documentclass[aop,MSNbibl,seceqn,dvips]{arximspdf}
\makeatletter
   \@ifpackageloaded{graphicx}{}{\usepackage{graphicx}}
\makeatother
\usepackage{mathrsfs} %\mathscr

\doi{10.1214/14-AOP976} % kopijuoti is laisko
\volume{44}
\issue{1}
\pubyear{2016}
\firstpage{360}
\lastpage{398}
\docsubty{FLA}

\makeatletter
\newcommand{\rrvert}{\vert}
\newcommand{\rrVert}{\Vert}
\newcommand{\llvert}{\vert}
\newcommand{\llVert}{\Vert}
\newtheorem{teo}{Theorem}[section]
\newtheorem{lem}[teo]{Lemma}
\newtheorem{prop}[teo]{Proposition}
\newproclaim{defn}{Definition}[section]
\newproclaim{ass}{Assumption}[section]
\newproclaim{rmk}{Remark}[section]
\newcommand{\cS}{\mathcal{S}}
\newcommand{\bF}{\mathbb{F}}
\newcommand{\bR}{\mathbb{R}}
\newcommand{\sF}{\mathscr{F}}
\newcommand{\sJ}{\mathscr{J}}
\newcommand{\sL}{\mathscr{L}}
\newcommand{\sS}{\mathscr{S}}
\newcommand{\sA}{\mathscr{A}}
\newcommand{\sP}{\mathscr{P}}
\makeatother

\begin{document}
\begin{frontmatter}

\title{On the Cauchy problem for backward stochastic partial
differential equations in H\"{o}lder spaces\thanksref{T1}}
\runtitle{H\"{o}lder solution of backward SPDEs}

\begin{aug}
\author[A]{\fnms{Shanjian}~\snm{Tang}\corref{}\ead[label=e1]{sjtang@fudan.edu.cn}}
\and
\author[A]{\fnms{Wenning}~\snm{Wei}\ead[label=e2]{wnwei@fudan.edu.cn}}
\runauthor{S. Tang and W. Wei}
\affiliation{Fudan University}
%\dedicated{}
\address[A]{Department of Finance and Control Sciences\\
Institute of Mathematical Finance\\
School of Mathematical Sciences\\
Fudan University\\
Shanghai 200433\\
China\\
\printead{e1}\\
\phantom{E-mail: }\printead*{e2}}
%\address[]{\\\printead{}}
\end{aug}
\thankstext{T1}{Supported in part by the Natural Science Foundation of
China (Grants
\#10325101 and \#11171076), Basic Research
Program of China (973 Program, No. 2007CB814904), WCU (World Class
University) Program
through the Korea Science and Engineering Foundation funded by the
Ministry of
Education, Science and Technology (R31-20007), the Science Foundation
for Ministry of
Education of China (No. 20090071110001), and the Chang Jiang Scholars
Programme.}

% HISTORY:
\received{\smonth{4} \syear{2013}}
\revised{\smonth{9} \syear{2014}}
%\accepted{\smonth{} \syear{}}

% ABSTRACT
%
\begin{abstract}
This paper is concerned with solution in H\"{o}lder spaces of the Cauchy
problem for linear and semi-linear
backward stochastic partial differential equations (BSPDEs) of
super-parabolic type. The pair of unknown
variables are viewed as deterministic spatial functionals which take
values in Banach spaces of random (vector)
processes. We define suitable functional H\"{o}lder spaces for them and
give some inequalities among
these H\"{o}lder norms. The existence, uniqueness as well as the
regularity of solutions are proved for
BSPDEs, which contain new assertions even on deterministic PDEs.
\end{abstract}

% KEYWORDS
\begin{keyword}[class=AMS]
\kwd[Primary ]{60H15}
%\kwd{}
\kwd[; secondary ]{35R60}
\end{keyword}
\begin{keyword}
\kwd{Backward stochastic partial differential equations}
\kwd{backward stochastic differential equations}
\kwd{H\"{o}lder space}
\kwd{heat potential}
\end{keyword}
\end{frontmatter}

%s1 #&#
\section{Introduction}\label{sec1}
In this paper, we consider the Cauchy problem for backward stochastic
partial differential equations (BSPDEs, for short) of super-parabolic type:
%e1.1 #&#
%
\begin{equation}
\label{bspde} \cases{ -du(t,x)= \bigl[a^{ij}(t,x)\,\partial^2_{ij}
u(t,x)+b^i(t,x)\,\partial_i u(t,x)
\vspace*{3pt}\cr
\hspace*{59pt}{}
+c(t,x)u(t,x)+f(t,x) +\sigma^l(t,x)v_l(t,x) \bigr] \,dt
\vspace*{3pt}\cr
\hspace*{56pt}{}
-v_l(t,x) \,dW^l_t,\qquad (t,x)\in[0,T)\times \bR^n;
\vspace*{5pt}\cr
u(T,x)=\Phi(x),\hspace*{76pt} x\in\bR^n.}
\end{equation}
Here, $T>0$ is fixed, $W=\{W_t\dvtx t\in[0,T]\}:=(W^1,\ldots, W^d)'$ is a
$d$-dimensional standard Brownian motion defined on some filtered
complete probability space $(\Omega,\sF,\bF,P)$ with $\bF:=\{\sF
_t\dvtx  t\in[0,T]\}$ being the augmented natural filtration generated by
$W$, $a:=(a^{ij})_{n\times n}$ is a symmetric and positive
matrix-valued deterministic functions of the time and space variable
$(t,x)$, $b:=(b^1,\ldots, b^n)'$ and $\sigma:=(\sigma^1, \ldots,\sigma^d)'$ are random vector fields, and $c, f$, and terminal term
$\Phi$ are scalar-valued random fields. Denote by $\sP$ the
predictable $\sigma$-algebra generated by~$\bF$. Here and after, we
use the Einstein summation convention, the prime denotes the transpose
of a vector or a matrix, and denote
\[
\partial_s:=\frac{\partial}{\partial s}, \qquad\partial_i:=
\frac
{\partial}{\partial x_i}, \qquad\partial^2_{ij}:=\frac{\partial
^2}{\partial x_i\,\partial x_j}.
\]
Our aim is to find a pair of random fields $(u,v)\dvtx [0,T]\times\Omega
\times\bR^n\rightarrow\bR\times\bR^d$ in suitable H\"{o}lder
spaces such that BSPDE (\ref{bspde}) is satisfied in some sense, and
to study the regularity of $(u,v)$, particularly in the space variable $x$.

As a mathematically natural extension of backward stochastic
differential equations (BSDEs) (see, e.g., \cite{Bismut73,Bismut76,Bismut78,ParPeng90,KarouiPengQuenez97,YongBook99}),
BSPDEs arise in many applications of probability theory and stochastic
processes. For instance, in the optimal control problem of stochastic
differential equations (SDEs) with incomplete information or stochastic
partial differential equations (SPDEs), a linear BSPDE arises as the
adjoint equation of SPDEs (or the Duncan--Mortensen--Zakai filtration
equation) to formulate the maximum principle (see, e.g., \cite{Bensoussan1983,Bensoussan1992,Tang1998,Tang1998conference,Zhouxunyu1992,Zhouxunyu1993}).
In the study of controlled non-Markovian SDEs by Peng \cite
{Peng1992}, the so-called stochastic Hamilton--Jacobi--Bellman
equation is a class of fully nonlinear BSPDEs. Solution of
forward--backward stochastic differential equation (FBSDE) with random
coefficients is also associated to that of a quasi-linear BSPDE, which
gives the stochastic Feynman--Kac formula (see, e.g., \cite{MaYong1997}).

Weak and strong solutions of linear BSPDEs have already received an
extensive attention in literature. Strong solution in the Sobolev space
$W^{m,2}$ is referred to, for example, \cite
{DuMeng2009,du2012strong,dokuchaev2012backward,du2013linear,HuMaYong2002,HuPeng1991,MaYong1997,MaYong1999,QiuTangYou2011},
and in $L^p$ [$p\in(1, \infty)$] is referred to, for example, \cite
{DuQiuTang2010}. The theory of linear BSPDEs in Sobolev spaces is
rather complete now. Qiu and Tang \cite{QiuTang2011} further discuss
the maximum principle of BSPDEs in a domain. It is quite natural to
consider now the H\"{o}lder solution of BSPDEs. We note that Tang \cite
{Tang2005} discusses the existence and uniqueness of a classical
solution to semi-linear BSPDE using a probabilistic approach. However,
the coefficients are required to be $k$-times (with $k\geq2+\frac
{n}{2}$) continuously differentiable in the spatial variable $x$, which
is much higher than the necessary regularity on the coefficients known
in the theory of deterministic PDEs. In this paper, the pair of unknown
variables are viewed as deterministic spatial functionals which take
values in Banach spaces of random (vector) processes. We discuss BSPDE
(\ref{bspde}) in H\"{o}lder spaces, using the methods of deterministic
PDEs (see Gilbarg and Trudinger \cite{Gilbarg2001},
Lady{\v{z}}enskaja, Solonnikov and Ural'ceva
\cite{Ladyzenskaja1968}), and establish a\vadjust{\goodbreak} H\"older theory for
BSPDEs under the spatial H\"older-continuity assumption on the
coefficients $a, b, c$ and $\sigma$. The paper seems to be the first
attempt at H\"older solution of BSPDEs.

As an alternative stochastic extension of deterministic second-order
pa\-rabolic equations, (forward) SPDEs have been studied in H\"
older spaces by Rozovskii \cite{rozovskiui1975stochastic} and
Mikulevicius \cite{Mikulevicius2000}. However, our BSPDE (\ref
{bspde}) is significantly different from an SPDE. Indeed, a BSPDE has
an additional unknown variable $v$ whose regularity is usually worse.
It serves in our BSPDE as the diffusion, but it is not a priori given.
Instead, it is endogenously determined by the given coefficients via a
martingale representation theorem. It is crucial to choose a suitable
H\"older space to describe its regularity.
In light of the functional H\"{o}lder space introduced by Mikulevicius
\cite{Mikulevicius2000} for discussing a SPDE, we define in Section~\ref{sec2}
the functional H\"{o}lder spaces such as $C^{m+\alpha}(\bR^n;\sS
^2_{\bF}[0,T])$ for $u$, and $C^{m+\alpha}(\bR^n;\sL^2_{\bF
}(0,T;\bR^d))$ for $v$. That is, we only discuss the continuity of the
unknown pair $(u,v)$ in $x$ by looking at $(u(\cdot,x), v(\cdot,x))$
as a\vspace*{1pt} functional stochastic process taking values in the space $\sS
^2_{\bF}[0,T]\times\sL^2_{\bF}(0,T;\bR^d)$.\vspace*{1pt}

We first study the following simpler BSPDE with space-invariant
coefficients $a$ and $\sigma$:
%e1.2 #&#
%
\begin{equation}
\label{bspdeS} \cases{ -du(t,x)= \bigl[a^{ij}(t)\,\partial^2_{ij}
u(t,x)+f(t,x)+\sigma ^l(t)v_l(t,x) \bigr]\,dt
\vspace*{3pt}\cr
\hspace*{56pt}{}
-v_l(t,x) \,dW^l_t,\qquad (t,x)\in[0,T)\times \bR^n,
\vspace*{5pt}\cr
u(T,x)=\Phi(x),\hspace*{76pt} x\in\bR^n.}
\end{equation}
Here, the coefficients $a^{ij}(\cdot)$ and $\sigma^l(\cdot)$
$(i,j=1,\cdot,n;  l=1,\cdot,d)$ do not depend on the space variable
$x$. The advantage of the simpler case is that the solution $(u,v)$
admits an explicit expression in terms of the terminal value $\Phi$
and the free term $f$ via their convolution with the heat potential. We
prove the existence and uniqueness result of this equation, and show
that $(u,v)(t, \cdot)\in C^{2+\alpha}\times C^{\alpha}$, and
$u(\cdot, x)\in C^{1/2}$, when $\Phi\in C^{1+\alpha}$ and
$f(t, \cdot)\in C^\alpha$. These regularity results are extended to
general space-variable BSPDE (\ref{bspde}) by the freezing
coefficients method and the standard continuity argument. Moreover,
when all the coefficients are deterministic, BSPDE (\ref{bspde})
becomes a deterministic PDE, and our results include new consequences
on a deterministic PDE.

The rest of the paper is organized as follows. In Section~\ref{sec2}, we define
some functional H\"{o}lder spaces, and recall analytical properties of
the heat potential. In Section~\ref{sec3}, we study the existence, uniqueness
and regularity of the solution of BSPDE (\ref{bspdeS}). In Section~\ref{sec4},
we extend the results in Section~\ref{sec3} to BSPDE (\ref{bspde}) via the
freezing coefficients method and the standard argument of continuity,
and discuss their consequences on a deterministic PDE. In Section~\ref{sec5}, we
discuss a semi-linear BSPDE.

%s2 #&#
\section{Preliminaries}\label{sec2}

%In this section, we define several functional H\"{o}lder spaces for
%Banach space valued functions, and recall analytical properties of the
%heat potential.

%s2.1 #&#
\subsection{Notation and H\"{o}lder spaces}\label{sec2.1}

Define the set of multi-indices
\[
\Gamma:=\bigl\{\gamma=(\gamma_1,\dots,\gamma_n)\dvtx
\gamma_1,\dots,\gamma _n\mbox{ are all nonnegative
integers}\bigr\}.
\]
For $\gamma\in\Gamma$ and $x=(x_1,\dots,x_n)\in\bR^n$, define
\[
\llvert \gamma\rrvert:=\sum_{i=1}^n
\gamma_i, \qquad D^{\gamma}:=\frac{\partial
^{\llvert  \gamma\rrvert  }}{
\partial x_1^{\gamma_1}\,\partial x_2^{\gamma_2}\cdots\partial
x_n^{\gamma_n}}.
\]
The inner product in an Euclidean space is denoted by $\langle\cdot
,\cdot\rangle$, and the norm by $\llvert  \cdot\rrvert  $.

The following are some spaces of random variables or stochastic
processes. For $p\in[1,+\infty]$, $L^p(\Omega, P, Y)=L^p(\Omega,
\sF_T, P, Y)$ is the Banach space of Hilbert space $Y$-valued random
variables $\xi$ on a complete probability space $(\Omega, \sF_T, P)$
with finite norm
\[
\llVert \xi\rrVert _{p,Y}=E \bigl[\llVert \xi\rrVert
_Y^p \bigr]^{1/p}, \qquad\llVert \xi\rrVert
_{\infty,Y}=\mathop{\operatorname{esssup}}_{\omega}\bigl\llVert \xi(
\omega)\bigr\rrVert _Y;
\]
$\sL^p_{\bF,P}(0,T;Y)$ is the Banach space of Hilbert space
$Y$-valued $\bF$-adapted processes~$f$ with finite norm
\[
\llVert f \rrVert _{\sL^p(Y)}:= E \biggl[\int_0^T
\bigl\llVert f(t) \bigr\rrVert _Y^p \,dt
\biggr]^{1/p}, \qquad\llVert f \rrVert _{\sL^\infty(Y)}:= \mathop{\operatorname{esssup}}_{(\omega, t)}\bigl\llVert f(\omega,t)\bigr\rrVert _Y;
\]
and
$\sS^p_{\bF, P}([0,T];Y)$ is the\vspace*{1pt} Banach space of Hilbert space
$Y$-valued $\bF$-adapted (path-wisely) continuous processes $f$ with
finite norm
\[
\llVert f \rrVert _{\sS^p(Y)}:=E \Bigl[\max_{t\in[0,T]}\bigl
\llVert f(t) \bigr\rrVert _Y^p \Bigr]^{1/p},
\qquad\llVert f \rrVert _{\sS^\infty(Y)}:= \llVert f \rrVert _{\sL^\infty}.
\]
If\vspace*{1pt} $Y=\bR$ or there is no confusion on the underlying Hilbert space
$Y$, we omit $Y$ in these notations and simply write $L^p(\Omega, P),
\sL^p_{\bF,P}(0,T),\sS^p_{\bF, P}[0,T]; \llVert  \xi\rrVert  _p, \allowbreak\llVert  f \rrVert  _{\sL^p}, \llVert  f \rrVert  _{\sS^p}, \ldots.$ Furthermore, the underlying
probability $P$ is omitted in these notations if there is no confusion
and we simply write $L^p(\Omega)$, $\sL^p_{\bF}(0,T)$, $\sS^p_{\bF
}[0,T], \ldots.$

Now we define our functional differentiable H\"{o}lder spaces. Let $m$
be a nonnegative integer, $\alpha\in(0,1)$ a constant, and $Y$ a
Banach space. $C^m(\bR^n, Y)$ is the Banach space of all $Y$-valued
continuous functionals defined on $\bR^n$ which are $m$-times
continuously differentiable (strongly in $Y$) with all the derivatives
up to order $m$ being bounded in $Y$, equipped with the norm
\[
\llVert \phi \rrVert _{m,Y}:=\sum^m_{k=0}[
\phi]_{k, Y},
\]
where
\[
[\phi]_{k,Y}:=\sum_{\llvert  \gamma\rrvert  =k}
\bigl[D^{\gamma}\phi\bigr]_{0, Y}, \qquad [\phi]_{0,Y}=
\llVert \phi \rrVert _{0,Y}:=\sup_{x\in\bR^n}\bigl\llVert
\phi(x)\bigr\rrVert _{Y}.
\]
$C^{m+\alpha}(\bR^n, Y)$ is the sub-space of all $\phi\in C^m(\bR
^n, Y)$ such that $[\phi]_{m+\alpha, Y}<+\infty$, where
\[
[\phi]_{\alpha, Y}:= \mathop{\sup_{x,y\in\bR^n}}_{x\neq y}
\frac{\llVert  \phi(x)-\phi(y)\rrVert
_Y}{\llvert  x-y \rrvert  ^{\alpha}},\qquad [\phi]_{m+\alpha, Y}:=\sum
_{\llvert  \gamma\rrvert  =m}\bigl[D^{\gamma}\phi\bigr]_{\alpha, Y}.
\]
For $\phi\in C^{m+\alpha}(\bR^n, Y)$, define the norm $\llVert  \phi \rrVert
_{m+\alpha, Y}:=\llVert  \phi \rrVert  _{m, Y}+[\phi]_{m+\alpha, Y}$.
If $Y=\bR$, these spaces, semi-norms, and norms are classical
differentiable and H\"{o}lder ones on $\bR^n$, and $Y$ will be omitted
in these notation and we simply write $C^m(\bR^n), C^{m+\alpha}(\bR
^n), [\cdot]_\alpha, [\cdot]_{m+\alpha}$, and $\llvert  \cdot\rrvert  _{m+\alpha
}$. In this paper, we shall take $Y=\bR, L^p(\Omega), \sL^p_{\bF
}(0,T;\bR^{\iota}), \sS^p_{\bF}([0,T])$ for $p\in[1, \infty]$.
Moreover, we use the following abbreviations:
\[
\llVert \cdot \rrVert _{m+\alpha,\sS^p}:=\llVert \cdot \rrVert _{m+\alpha,\sS^p_{\bF
}[0,T]},
\qquad\llVert \cdot \rrVert _{m+\alpha,\sL^p}:=\llVert \cdot \rrVert
_{m+\alpha,\sL
^p_{\bF}(0,T;\bR^{\iota})},
\]
and similar abbreviations for semi-norms.

For $L^p(\Omega, \sF_T, P,Y)$-valued functional $u$ defined on
$[0,T]\times\bR^n$, we denote its partial derivatives
in the space $L^p(\Omega, \sF_T, P,Y)$ by $\partial_t u:={\partial u\over\partial t}, \partial_i u:={\partial u\over\partial x_i},
\partial_{ij}^2u:={\partial^2 u\over\partial x_i\,\partial x_j}$, etc.

$C = C(\cdot, \ldots,\cdot)$ denotes a constant depending only on quantities
appearing in parentheses. In a given context, the same letter will (generally)
be used to denote different constants depending on the same set of
arguments.

It can be verified that
%e2.1 #&#
%
\begin{equation}
\label{remark21} [h\psi]_{\alpha,\sL^p}\leq[h]_{0,\sL^{\infty}}[
\psi]_{\alpha,\sL^p}+[h]_{\alpha,\sL^{\infty}}[\psi]_{0,\sL^p}
\end{equation}
for any $ (h, \psi)\in C^{\alpha}(\bR^n, \sL^{\infty}_{\bF
}(0,T;\bR^{\iota}))\times C^{\alpha}(\bR^n, \sL^p_{\bF}(0,T;\bR
^{\iota}))$.
This inequality will be used in Section~\ref{sec4}.

Similar to classical H\"{o}lder spaces of scalar- or finite-dimensional
vector-valued functions, we have the following interpolation inequalities.
%le2.1 #&#

\begin{lem}\label{interpolation}
For $\varepsilon>0$, there is $C=C(\varepsilon,\alpha)>0$ such that
for all $\psi\in C^{2+\alpha}(\bR^n, \sL^p_{\bF}(0,T;\bR^{\iota}))$
\begin{eqnarray*}
[\psi]_{2,\sL^p}&\leq&\varepsilon[\psi]_{2+\alpha,\sL^p}+C [\psi
]_{0,\sL^p},
\\
{[\psi]}_{1+\alpha,\sL^p}&\leq&\varepsilon[\psi]_{2+\alpha,\sL
^p}+C [
\psi]_{0,\sL^p},
\\
{[\psi]}_{1,\sL^p}&\leq&\varepsilon[\psi]_{2+\alpha,\sL^p}+C [\psi
]_{0,\sL^p},
\\
{[\psi]}_{\alpha,\sL^p}&\leq&\varepsilon[\psi]_{2+\alpha,\sL^p}+C [
\psi]_{0,\sL^p}.
\end{eqnarray*}
Analogous inequalities also hold for elements of the H\"{o}lder
functional space $C^{2+\alpha}(\bR^n, L^p(\Omega))$ or $C^{2+\alpha
}(\bR^n, \sS^p_{\bF}[0,T])$.
\end{lem}

The proof is similar to that of the interpolation inequalities in the
classical H\"{o}lder spaces in Gilbarg and Trudinger \cite{Gilbarg2001}, Lemma~6.32. It is omitted here.

%s2.2 #&#
\subsection{Linear BSPDEs}\label{sec2.2}

Consider the Cauchy problem of linear BSPDE~(\ref{bspde}) in
functional H\"{o}lder spaces. Denote by $\cS^n$ the totality of all
$n\times n$-symmetric matrices.
Assume that all the coefficients:
\begin{eqnarray*}
a\dvtx [0,T]\times\bR^n&\rightarrow& \cS^n,\qquad  b\dvtx [0,T]\times\Omega\times \bR^n\rightarrow \bR^n,
\\
c\dvtx [0,T]\times\Omega\times\bR^n&\rightarrow&\bR,\qquad \sigma\dvtx [0,T] \times\Omega\times\bR^n\rightarrow\bR^d,
\\
f\dvtx [0,T]\times\Omega\times\bR^n&\rightarrow&\bR,\qquad \Phi\dvtx \Omega \times \bR^n\rightarrow\bR,
\end{eqnarray*}
are random fields and jointly measurable, and are $\bF$-adapted or
$\sF_T$-measurable at each $x\in\bR^n$. We make the following assumptions.

%as2.1 #&#
%
\begin{ass}[(Super-parabolicity)]\label{superparab2}
There are two positive constants $\lambda$ and $\Lambda$ such that
\[
\lambda\llvert \xi \rrvert ^2\leq\bigl\langle a(t,x)\xi, \xi\bigr
\rangle\leq\Lambda\llvert \xi \rrvert ^2\qquad\forall(t, x, \xi)
\in[0,T]\times\bR^n\times\bR^n.
\]
\end{ass}

%as2.2 #&#
%
\begin{ass}[(Boundedness)]\label{boundedness}
The functionals
\begin{eqnarray*}
a &\in& C^{\alpha}\bigl(\bR^n, L^\infty\bigl(0,T;
\cS^n\bigr)\bigr),\qquad b\in C^{\alpha}\bigl(\bR ^n,
\sL^{\infty}_{\bF}\bigl(0,T;\bR^n\bigr)\bigr),
\\
c&\in& C^{\alpha}\bigl(\bR^n, \sL ^{\infty}_{\bF}(0,T)
\bigr),
\end{eqnarray*}
and $\sigma\in C^{\alpha}(\bR^n, \sL^{\infty}_{\bF}(0,T;\bR
^d))$. Also, $a,b,c$ and $\sigma$ are bounded, that is, there is
$\Lambda>0$ such that
$\llVert  a \rrVert  _{\alpha, L^{\infty}}+\llVert  b \rrVert  _{\alpha,\sL^{\infty}}+\llVert  c \rrVert
_{\alpha,\sL^{\infty}}+\llVert  \sigma \rrVert  _{\alpha,\sL^{\infty}}\leq
\Lambda$.
\end{ass}

Note that throughout the paper $a$ is assumed to be a deterministic
$\cS^n$-valued bounded function of the time--space variable $(t,x)$.

A classical solution to BSPDE~(\ref{bspde}) in H\"{o}lder spaces is
defined as follows.
%de2.1 #&#

\begin{defn}
Let $\Phi\in C^{1+\alpha}(\bR^n, L^2(\Omega))$ and $f\in C^{\alpha
}(\bR^n, \sL^2_{\bF}(0,T))$. We call $(u,v)$ a classical solution to
BSPDE (\ref{bspde}) if
\[
(u,v)\in C^{\alpha}\bigl(\bR^n, \sS^2_{\bF}[0,T]
\bigr)\cap C^{2+\alpha}\bigl(\bR ^n, \sL^2_{\bF}(0,T)
\bigr)\times C^{\alpha}\bigl(\bR^n, \sL^2_{\bF}
\bigl(0,T;\bR^d\bigr)\bigr),
\]
and for all $(t,x)\in[0,T]\times\bR^n$,
\begin{eqnarray*}
u(t,x) &=& \Phi(x)
+\int_t^T
\bigl[a^{ij}(s,x)\,\partial ^2_{ij}u(s,x)+b^i(s,x)\,\partial_iu(s,x)
\\
&&\hspace*{55pt}{}+c(s,x)u(s,x)+f(s,x)+\sigma^l(s,x)v_l(s,x) \bigr] \,ds
\\
&&{} -\int
_t^Tv_l(s,x)\,dW^l_s,
\qquad P\mbox{-a.s.}
\end{eqnarray*}
\end{defn}

For simplicity of notation, define
\begin{eqnarray*}
&& C^{\alpha}_{\sS^2}\cap C^{2+\alpha}_{\sL^2}\times
C^{\alpha}_{\sL^2}
\\
&&\qquad := C^{\alpha}\bigl(\bR^n;
\sS^2_{\bF}[0,T]\bigr)\cap C^{2+\alpha}\bigl(
\bR^n;\sL ^2_{\bF}(0,T)\bigr)
C^{\alpha}\bigl(\bR^n;\sL^2_{\bF}
\bigl(0,T;\bR^d\bigr)\bigr).
\end{eqnarray*}

%s2.3 #&#
\subsection{Estimates on the heat potential}\label{sec2.3}
Consider the heat equation:
%e2.2 #&#
%
\begin{equation}
\label{heatequ} \partial_t u(t,x)=a^{ij}(t)\,\partial^2_{ij}u(t,x), \qquad(t,x)\in [0,T]\times
\bR^n,
\end{equation}
where $a=(a^{ij})_{n\times n}\dvtx [0,T]\rightarrow\cS^n$ satisfies the
super-parabolic assumption.
Define
\[
G_{s,t}(x):=\frac{1}{(4\pi)^{n/2} (\operatorname{det} A_{s,t})^{1/2}} \exp{ \biggl(-\frac{1}{4}
\bigl(A^{-1}_{s,t} x, x \bigr) \biggr)}\qquad\forall0\leq t<s
\leq T,
\]
where $A_{s,t}:=\int_t^s a(r) \,dr$. In a straightforward way, we have
%e2.3 #&#
%e2.4 #&#
%
\begin{eqnarray}
\partial_sG_{s,t}(x)&=& a^{ij}(s)\,
\partial^2_{ij}G_{s,t}(x),\qquad s>t;
\nonumber\\[-8pt]\\[-8pt]\nonumber
\partial_tG_{s,t}(x)&=&-a^{ij}(t)\,\partial^2_{ij}G_{s,t}(x),\qquad s>t.
\end{eqnarray}

%re2.1 #&#
%
\begin{rmk} From
Lady{\v{z}}enskaja, Solonnikov and Ural'ceva (\cite{Ladyzenskaja1968}, (1.7)) and (2.5) of Chapter~IV, we have for $\gamma\in\Gamma$,
%e2.5 #&#
%
\begin{equation}
\label{poten1} \int_{\bR^n}D^{\gamma}G_{s,t}(x)\,dx=
\cases{ 1, &\quad$\gamma=0$,
\vspace*{3pt}\cr
0, &\quad$\llvert \gamma\rrvert >0$}
\end{equation}
and there are $C=C(\lambda,\Lambda,\gamma,n,T)$ and $c\in(0, \frac
{1}{4})$ such that
%e2.6 #&#
%
\begin{equation}
\label{poten2} \bigl\llvert D^{\gamma}G_{s,t}(x)\bigr\rrvert \leq
C(s-t)^{-({n+\llvert  \gamma\rrvert  })/{2}}\exp{ \biggl(-c\frac{\llvert  x\rrvert  ^2}{s-t} \biggr)}\qquad\forall s>t.
\end{equation}
Furthermore, we have
%e2.7 #&#
%
\begin{equation}
\label{poten3} \int_0^s\bigl\llvert
D^{\gamma}G_{s,t}(x)\bigr\rrvert \,dt\leq C\llvert x\rrvert
^{-(n+\llvert  \gamma\rrvert  )+2}\int_0^\infty r^{{(n+\llvert  \gamma\rrvert  )/2}+2}
\exp(-c r) \,dr
\end{equation}
for any $s\in[0,T]$ and $x\neq0$,
and
%e2.8 #&#
%
\begin{equation}
\label{poten4} \quad\int_{\bR^n}\bigl\llvert D^{\gamma}G_{s,t}(x)
\bigr\rrvert \llvert x\rrvert ^{\alpha} \,dx\leq C (s-t)^{({\alpha-\llvert  \gamma\rrvert  })/{2}}\int
_{\bR^n} \llvert x\rrvert ^\alpha\exp \bigl(-c\llvert x
\rrvert ^2\bigr) \,dx
\end{equation}
for $s>t$ and $\alpha\in(0,1)$.
\end{rmk}

The following lemmas will be used to derive a priori H\"{o}lder
estimates in Section~\ref{sec3}.

From Mikulevicius \cite{Mikulevicius2000}, Lemma 4, we have:
%le2.2 #&#

\begin{lem}\label{estpoten1}
For any multi-index $\llvert  \gamma\rrvert  =2$, there exists $C=C(\lambda,\Lambda,\gamma,n,T)$ such that for $0\leq\tau\leq s\leq T$ and $\eta>0$,
\[
\int_\tau^s\biggl\llvert \int
_{\llvert  y \rrvert  \leq\eta}D^{\gamma}G_{s,t}(y) \,dy\biggr\rrvert \,dt
=\int_\tau^s\biggl\llvert \int
_{\llvert  y \rrvert  \geq\eta}D^{\gamma}G_{s,t}(y) \,dy\biggr\rrvert \,dt
\leq C.
\]
\end{lem}

%le2.3 #&#
%
\begin{lem}\label{estpoten3}
Let $\eta>0$ be a constant. Then for $\gamma\in\Gamma$ such that
$\llvert  \gamma\rrvert  =2$, there is a constant $C=C(\lambda,\Lambda,\gamma,\alpha,n,T)$ such that
\[
\int_{B_{\eta}(0)}\sup_{\tau\leq s}\int
_\tau^s\bigl\llvert D^{\gamma
}G_{s,t}(y)
\bigr\rrvert \llvert y \rrvert ^{\alpha} \,dt \,dy\leq C \eta^{\alpha}.
\]
\end{lem}

\begin{pf}
In view of (\ref{poten3}), we have
\[
\int_{B_{\eta}(0)}\sup_{\tau\leq s}\int
_\tau^s\bigl\llvert D^{\gamma
}G_{s,t}(y)
\bigr\rrvert \llvert y \rrvert ^{\alpha} \,dt \,dy \leq C\int
_{B_{\eta}(0)}\llvert y \rrvert ^{-n+\alpha}\,dy \leq C
\eta^{\alpha}.
\]\upqed
\end{pf}

%le2.4 #&#
%
\begin{lem}\label{estpoten4}
For any $x,\bar{x}\in\bR^n$ and $\gamma\in\Gamma$ such that
$\llvert  \gamma\rrvert  =2$, we have
\[
\int_{\llvert  y-x \rrvert  >\eta}\sup_{\tau\leq s}\int
_\tau^s \bigl\llvert D^{\gamma}G_{s,t}(x-y)-D^{\gamma}G_{s,t}(
\bar{x}-y)\bigr\rrvert \llvert \bar{x}-y\rrvert ^{\alpha} \,dt \,dy \leq C
\eta^{\alpha},
\]
where $\eta:=2\llvert  x-\bar{x}\rrvert  $ and $C=C(\lambda,\Lambda,\gamma,\alpha,n,T)$.
\end{lem}

\begin{pf} Define ${\widetilde x}:=x+2(\bar x-x)$. Let $\xi$ be any
point on the segment joining $x$ and $\bar{x}$.
For $\llvert  y-x \rrvert  >\eta$, we have
\begin{eqnarray*}
\llvert \xi-x\rrvert &\le& \tfrac{1}{2}\llvert x-{\widetilde x}\rrvert \le
\tfrac{1}{2}\llvert x-y \rrvert,
\\
\tfrac{1}{2}\llvert x-y
\rrvert &\leq& \llvert \xi-y\rrvert =\bigl\llvert (x-y)+(\xi-x)\bigr\rrvert \leq
\tfrac{3}{2}\llvert x-y \rrvert.
\end{eqnarray*}
In view of (\ref{poten2}) and (\ref{poten3}), we have
\begin{eqnarray*}
&&\int_{\llvert  y-x \rrvert  >\eta}\sup_{\tau\leq s}\int
_\tau^s\bigl\llvert D^{\gamma
}G_{s,t}(x-y)-D^{\gamma}G_{s,t}(
\bar{x}-y)\bigr\rrvert \llvert \bar{x}-y\rrvert ^{\alpha} \,dt \,dy
\\
&&\qquad \leq C\eta\int_{\llvert  y-x \rrvert  >\eta}\sup_{\tau\leq s}\int
_\tau^s \int_0^1
\bigl\llvert \partial_xD^{\gamma}G_{s,t} \bigl(r
\bar{x}+(1-r)x-y\bigr)\bigr\rrvert \,dr\llvert x-y \rrvert ^{\alpha} \,dt \,dy
\\
&&\qquad \leq C\eta\int_{\llvert  y-x \rrvert  >\eta}\int_0^T  \int_0^1 t^{-(n+3)/{2}}
\\
&&\hspace*{118pt}{}\times \exp \biggl(-c
\frac{\llvert  r\bar{x}+(1-r)x-y\rrvert  ^2}{t} \biggr)\,dr\llvert x-y \rrvert ^{\alpha
} \,dt \,dy
\\
&&\qquad \leq C\eta\int_0^T  \int_{\llvert  y-x \rrvert  >\eta}t^{-(n+3)/{2}}
\exp \biggl(-c\frac{\llvert  x-y \rrvert  ^2}{t} \biggr)\llvert x-y \rrvert ^{\alpha} \,dt \,dy
\\
&&\qquad \leq C\eta\int_{\llvert  y-x \rrvert  >\eta}\llvert x-y \rrvert ^{-n-1+\alpha} \,dy
\leq C\eta^{\alpha}.
\end{eqnarray*}\upqed
\end{pf}

%s3 #&#
\section{BSPDE with space-invariant coefficients $a$ and \texorpdfstring{$\sigma$}{$sigma$}}\label{sec3}
Consider the following linear BSPDE:
%e3.1 #&#
%
\begin{equation}
\label{simplelinearbspde}
\cases{ -du(t,x)= \bigl[a^{ij}(t)\,\partial^2_{ij}u(t,x)+f(t,x)+
\sigma ^l(t)v_l(t,x) \bigr] \,dt
\vspace*{3pt}\cr
\hspace*{56pt}{} -v_l(t,x)
\,dW^l_t,\qquad (t,x)\in[0,T)\times\bR^n,
\vspace*{5pt}\cr
u(T,x) = \Phi(x),\hspace*{77pt} x\in\bR^n,}
\end{equation}
where $a:=(a^{ij})_{n\times n}\dvtx [0,T]\rightarrow\cS^n$ is Borel
measurable and $\sigma:=(\sigma^1,\ldots,\sigma^d)'\dvtx\break  \Omega\times
[0,T]\rightarrow\bR^d$ is $\bF$-adapted. It is simpler than
BSPDE~(\ref{bspde}), for both coefficients $a$ and $\sigma$ are
assumed to be independent of the space variable $x$ (hence not varying
with the space variable, and hereafter called \textit{space-invariant}).
For this case, both Assumptions \ref{superparab2} and \ref
{boundedness} can be combined into the following one.

%as3.1 #&#
%
\begin{ass}\label{superparab} $a\in\sL^\infty(0,T;\cS^n)$ and
$\sigma\in\sL^\infty(0,T;\bR^d)$.
There are two positive constants $\lambda$ and $\Lambda$ such that
$\lambda\llvert  \xi \rrvert  ^2\leq\langle a(t)\xi, \xi\rangle\leq\Lambda\llvert  \xi
\rrvert  ^2$ for any $(t,\xi)\in[0,T]\times\bR^n$ and $\llVert  \sigma \rrVert  _{\sL
^{\infty}}\leq\Lambda$.
\end{ass}

The special structural assumption on both coefficients $a$ and $\sigma
$ allows us to give an explicit expression of the adapted solution
$(u,v)$ to BSPDE~(\ref{simplelinearbspde}). To see this point, let
us look at the respective contributions of both coefficients $a$ and
$\sigma$ to the solution $(u,v)$ of BSPDE~(\ref{simplelinearbspde}).

Define
%e3.2 #&#
%
\begin{equation}
\label{nBM}\widetilde{W}_t:=-\int_0^t
\sigma(s) \,ds+W_t, \qquad t\in[0,T]
\end{equation}
and the equivalent probability $Q$ by
%e3.3 #&#
%
\begin{equation}
\label{nProb}dQ:=\exp \biggl( \int_0^T\bigl
\langle\sigma (t), dW_t\bigr\rangle-\frac{1}{2} \int
_0^T \bigl\llvert \sigma(t)\bigr\rrvert
^2\,dt \biggr) \,dP.
\end{equation}
It can be verified that $\widetilde{W}$ is a standard Brownian motion
on $(\Omega,\sF_T,\bF,Q)$. BSPDE~(\ref{simplelinearbspde}) is
written into the following form:
%e3.4 #&#
%e3.5 #&#
%
\begin{eqnarray}\label{equivform}
u(t,x)&=& \Phi(x)+\int_t^T
\bigl[a^{ij}(r)\,\partial ^2_{ij}u(r,x)+f(r,x)
\bigr] \,dr
\nonumber\\[-8pt]\\[-8pt]\nonumber
&&{} -\int_r^Tv_l(r,x) \,d\widetilde
W^l_t, \qquad(t,x)\in[0,T]\times \bR^n.
\end{eqnarray}
Furthermore, we have for $(t,x)\in[0,T]\times\bR^n$
%e3.6 #&#
%
\begin{equation}
u(t,x)= E_Q^{\sF_t} \biggl[\Phi(x)+\int_t^T
\bigl(a^{ij}(r)\,\partial ^2_{ij}u(r,x)+f(r,x)
\bigr) \,dr \biggr].
\end{equation}
Since $a$ is deterministic and $Q$ does not depend on the space
variable $x$ (in view of Assumption~\ref{superparab}), we have for
$(t,x)\in[0,T]\times\bR^n$
\[
u(t,x)=E_Q^{\sF_t}\Phi(x)+\int_t^T
\bigl[a^{ij}(r)\,\partial ^2_{ij}
\bigl(E_Q^{\sF_t}u(r,x) \bigr)+E_Q^{\sF_t}f(r,x)
\bigr] \,dr, \qquad\mbox{a.s.} %
\]
Note that it is the integral on $[t,T]$ with respect to $r$ of the
following backward parabolic equation with $U(r,x;t):=E_Q^{\sF_t}u(r,x)$:
\[
\cases{ -\partial_r U(r,x;t)= a^{ij}(r)\,\partial^2_{ij}U(r,x;t)+E_Q^{\sF
_t}f(r,x),
&\quad $(r,x)\in[t,T)\times\bR^n$,
\vspace*{3pt}\cr
U(T,x)= E_Q^{\sF_t}
\Phi(x), &\quad$x\in\bR^n$.}
\]
Define for the convolution of the heat potential $G_{s,t}$ with a
functional $\phi$ defined on~$\bR^n$ and a functional $\psi$ defined
on~$[0,T]\times\bR^n$ as follows:  for $x\in\bR^n$,
%e3.7 #&#
%e3.8 #&#
%
\begin{eqnarray}
\label{note} R^s_t\phi(x)&:=&\int_{\bR^n}G_{s,t}(x-y)
\phi(y)\,dy\qquad\forall s>t,
\nonumber\\[-8pt]\\[-8pt]\nonumber
R^s_t\psi(\cdot) (x)&:=&\int_{\bR^n}G_{s,t}(x-y)
\psi(\cdot,y)\,dy\qquad\forall s>t.
\end{eqnarray}
It is well known that the solution of the last PDE has the following
representation: almost surely:
\begin{eqnarray}
U(r,x;t)&=& R_r^T \bigl(E_Q^{\sF_t}
\Phi \bigr) (x)+ \int_r^TR_r^s
\bigl(E_Q^{\sF_t}f(s,\cdot) \bigr) (x) \,ds,\nonumber
\\
\eqntext{(r,x)\in[t,T] \times\bR^n.}
\end{eqnarray}
Setting $r=t$, we have almost surely
\[
u(t,x)= R_t^T \bigl(E_Q^{\sF_t}\Phi
\bigr) (x)+ \int_t^TR_t^s
\bigl(E_Q^{\sF_t}f(s,\cdot) \bigr) (x) \,ds, \qquad(t,x)
\in[0,T]\times \bR^n. %
\]
It is easy to see that $\{E_Q^{\sF_t}\Phi(x), t\in[0,T]\}$ and $\{
E_Q^{\sF_t}f(s,x), t\in[0,T]\}$ are uniquely characterized by
\begin{eqnarray}
E_Q^{\sF_t}\Phi(x)&=&\varphi(t;x),\qquad E_Q^{\sF_t}f(s,x)=Y(t;s,x),\nonumber
\\
\eqntext{(t,x)\in[0,T]\times\bR^n, \mbox{ a.s.},}
\end{eqnarray}
where $(\varphi(\cdot;x), \psi(\cdot;x))$ and $(Y(\cdot;\tau,x),
g(\cdot;\tau,x))$ are the unique adapted solution of the following
two parameterized BSDEs:
%e3.9 #&#
%
\begin{eqnarray}\label{bsde01}
\varphi(t;x)&=&\Phi(x)+\int_t^T
\sigma^l(r)\psi_l(r;x) \,dr-\int_t^T
\psi_l(r;x) \,dW^l_r,
\nonumber\\[-8pt]\\[-8pt]
\eqntext{t\in[0, T]}
\end{eqnarray}
and
%e3.10 #&#
%
\begin{eqnarray}\label{bsde02}
Y(t;\tau,x)&=&f(\tau,x)+\int_t^\tau
\sigma^l(r)g_l(r;\tau,x) \,dr-\int_t^\tau
g_l(r;\tau,x) \,dW^l_r,
\nonumber\\[-8pt]\\[-8pt]
\eqntext{t\in[0,\tau],}
\end{eqnarray}
respectively. In this way, we have the desired representation of
$(u,v)$: for $(t,x)\in[0,T]\times\bR^n$,
\[
u(t,x)=R_t^T\varphi(x)+\int_t^TR_t^sY(t;r,
\cdot) (x) \,dr, %
\]
and further we expect from the linear structure of our BSPDE that
\[
v(t,x)=R_t^T\psi(t, \cdot) (x)+\int
_t^TR_t^sg(t;r,
\cdot) (x) \,dr, %
\]
which is stated as the subsequent Theorem~\ref{representationuv}.

The rest of the section is structured as follows. In Section~\ref{sec3.1}, we
prove the above explicit expression for the classical solution $(u,v)$
to BSPDE (\ref{simplelinearbspde}) in terms of the terminal term
$\Phi$ and the free term $f$. In Section~\ref{sec3.2}, we derive the a priori
H\"{o}lder estimates. Finally, in Section~\ref{sec3.3}, we prove the existence
and uniqueness result of classical solution to BSPDE (\ref
{simplelinearbspde}).

%s3.1 #&#
\subsection{Explicit expression of $(u,v)$}\label{sec3.1}

%le3.1 #&#
%
\begin{lem}\label{repreu}
Suppose that Assumption \ref{superparab} holds, $\Phi\in C^{1+\alpha
}(\bR^n, L^2(\Omega))$ and $f\in C^{\alpha}(\bR^n, \sL^2_{\bF
}(0,T))$. If $(u,v)\in C^{\alpha}_{\sS^2}\cap C^{2+\alpha}_{\sL
^2}\times C^{\alpha}_{\sL^2}$ is the\vspace*{2pt} classical solution of BSPDE
(\ref{simplelinearbspde}), then for all $(t,x)\in[0,T]\times\bR
^n$, we have almost surely
\begin{eqnarray*}
u(t,x)&=&R^T_t\Phi(x)+\int_t^T
\bigl[R^s_tf(s) (x)+\sigma ^l(s)R^s_tv_l(s)
(x) \bigr]\,ds
\\
&&{} -\int_t^T R^s_tv_l(s)
(x) \,dW^l_s.
\end{eqnarray*}
\end{lem}

\begin{pf}
For fixed $(t,x)\in[0,T]\times\bR^n$ and $s\in(t,T]$, using It\^
{o}'s formula, we have
%e3.11 #&#
%e3.12 #&#
%e3.13 #&#
%e3.14 #&#
%e3.15 #&#
%
\begin{eqnarray}\label{repreu1}
&&G_{s,t}(x-y)u(s,y)\nonumber
\\
&&\qquad = G_{T,t}(x-y)u(T,y)-\int_s^TG_{r,t}(x-y)
\,du(r,y)\nonumber
\\
&&\quad\qquad{} -\int_s^Tu(r,y) \,dG_{r,t}(x-y)\nonumber
\\
&&\qquad = G_{T,t}(x-y)\Phi(y)
\\
&&\quad\qquad{} +\int_s^T
\bigl[G_{r,t}(x-y)f(r,y)+\sigma ^l(r)G_{r,t}(x-y)v_l(r,y)
\bigr] \,dr\nonumber
\\
&&\quad\qquad{} -\int_s^TG_{r,t}(x-y)v_l(r,y)
\,dW^l_r\nonumber
\\
&&\quad\qquad{} -\int_s^T \bigl[a^{ij}(r)\,\partial^2_{ij}G_{r,t}(x-y)u(r,y)
-G_{r,t}(x-y)a^{ij}(r)\,\partial^2_{ij}u(r,y)
\bigr] \,dr.\hspace*{-8pt}\nonumber
\end{eqnarray}
A direct computation shows that
\[
\int_{\bR^n}\int_s^T
\bigl[a^{ij}(r)\,\partial^2_{ij}G_{r,t}(x-y)
u(r,y) -G_{r,t}(x-y)a^{ij}(r)\,\partial^2_{ij}u(r,y)
\bigr] \,dr \,dy=0.
\]
Stochastic Fubini theorem (see Da Prato \cite{DaPrato1992}, Theorem~4.18) gives that
\[
\int_{\bR^n}\int_s^TG_{r,t}(x-y)v_l(r,y)
\,dW^l_r \,dy=\int_s^T  \int_{\bR^n}G_{r,t}(x-y)v_l(r,y) \,dy
\,dW^l_r.
\]
Thus, integrating w.r.t. $y$ over $\bR^n$ both sides of (\ref
{repreu1}), we have
%e3.16 #&#
%e3.17 #&#
%
\begin{eqnarray}\label{repreu2}
&&\int_{\bR^n}G_{s,t}(x-y)u(s,y) \,dy\nonumber
\\
&&\qquad = R^T_t\Phi(x)+\int_s^T
\bigl[R^r_tf(r) (x)+\sigma ^l(r)R^r_tv_l(r)
(x) \bigr] \,dr
\\
&&\quad\qquad{} -\int_s^TR^r_tv_l(r)
(x) \,dW^l_r.\nonumber
\end{eqnarray}
In what follows, we\vspace*{1pt} compute the limit of each part of (\ref
{repreu2}) as $s\rightarrow t$.

Since $u\in C^{\alpha}(\bR^n, \sS^2_{\bF}[0,T])$, we have from
estimates (\ref{poten1}) and (\ref{poten2}) on the heat potential that
%e3.18 #&#
%e3.19 #&#
%e3.20 #&#
%e3.21 #&#
%
\begin{eqnarray}
\label{repreu3} \qquad&& E\biggl\llvert \int_{\bR^n}G_{s,t}(x-y)
u(s,y) \,dy-u(t,x)\biggr\rrvert ^2\nonumber
\\
&&\qquad = E\biggl\llvert \int_{\bR^n}G_{s,t}(x-y)
\bigl[u(s,y)-u(t,x)\bigr] \,dy\biggr\rrvert ^2\nonumber
\\
&&\qquad \leq C E\int_{\bR^n}G_{s,t}(x-y) \bigl\llvert
u(s,y)-u(t,x)\bigr\rrvert ^2 \,dy
\\
&&\qquad \leq C E\int_{\bR^n}\exp{ \bigl(-c\llvert z \rrvert
^2 \bigr)} \bigl\llvert u(s,x-\sqrt{s-t}z)-u(t,x)\bigr\rrvert
^2 \,dz % \leq &C E\left[\int_{\bR^n}e^{-c\left\vert z \right\vert ^2}\bigl\left\vert u(s,x+
%\sqrt{s-t}z)-u(s,x)\bigr\left\vert ^2 \,dz+\bigl\left\vert u(s,x)-u(t,x)\bigr\left\vert ^2\right]\\
% \leq &C \sup_zE\bigl\left\vert u(s,x+\sqrt{s-t}z)-u(s,x)\bigr\left\vert ^2+C E
%\bigl\left\vert u(s,x)-u(t,x)\bigr\left\vert ^2\\
\longrightarrow 0,\nonumber
\\
\eqntext{\mbox{as }s\downarrow t.}
\end{eqnarray}
In view of (\ref{poten1}), we have
\begin{eqnarray*}
&& E\biggl\llvert \int_t^s\int
_{\bR^n}G_{r,t}(x-y) v_l(r,y) \,dy
\,dW^l_r\biggr\rrvert ^2
\\
&&\qquad =E\int_t^s\biggl\llvert \int
_{\bR^n}G_{r,t}(x-y) \bigl[v(r,y)-v(r,x)\bigr]
\,dy+v(r,x)\biggr\rrvert ^2 \,dr
\\
&&\qquad \leq 2 E\int_t^s\biggl\llvert \int
_{\bR^n}G_{r,t}(x-y) \bigl[v(r,y)-v(r,x)\bigr] \,dy
\biggr\rrvert ^2 \,dr+2E\int_t^s
\bigl\llvert v(r,x) \bigr\rrvert ^2 \,dr
\\
&&\qquad \leq C E \int_t^s \int_{\bR^n}G_{r,t}(x-y)
\bigl\llvert v(r,y)-v(r,x)\bigr\rrvert ^2 \,dy \,dr +2E\int
_t^s\bigl\llvert v(r,x) \bigr\rrvert
^2 \,dr
\\
&&\qquad \leq C\int_{\bR^n}\sup_{r\in[t,s]}
G_{r,t}(x-y) \llvert x-y \rrvert ^{2\alpha
}\,dy \cdot \sup
_{y\in\bR^n} E \int_t^s
\frac
{\llvert  v(r,y)-v(r,x)\rrvert  ^2}{\llvert  x-y \rrvert  ^{2\alpha}}\,dr
\\
&&\quad\qquad{} +2E\int_t^s \bigl\llvert v(r,x) \bigr
\rrvert ^2\,dr.
\end{eqnarray*}
Since $v\in C^{\alpha}(\bR^n, \sL^2_{\bF}(0,T;\bR^d))$, we have
\[
\lim_{s \downarrow t}E\int_t^s\bigl
\llvert v(r,x) \bigr\rrvert ^2 \,dr=0,\qquad\lim_{s
\downarrow t}
\sup_{y\in\bR^n}E\int_t^s
\frac
{\llvert  v(r,y)-v(r,x)\rrvert  ^2}{\llvert  x-y \rrvert  ^{2\alpha}} \,dr=0.
\]
In view of subsequent Lemma \ref{zz}, we obtain
%e3.22 #&#
%
\begin{equation}
\label{repreu4} \lim_{s \downarrow t}E\biggl\llvert \int
_t^s \int_{\bR
^n}G_{r,t}(x-y)
v_l(r,y) \,dy \,dW^l_r\biggr\rrvert
^2=0.
\end{equation}

In a similar way, we have
%e3.23 #&#
%
\begin{equation}
\label{repreu5} \lim_{s \downarrow t} E\biggl\llvert \int
_t^s \bigl[R^r_tf(r)
(x)+\sigma ^l(r)R^r_tv_l(r) (x)
\bigr] \,dr\biggr\rrvert ^2=0.
\end{equation}
Letting $s\rightarrow t$ in equality (\ref{repreu2}), we have the
desired result from (\ref{repreu3}), (\ref{repreu4}) and~(\ref
{repreu5}).
\end{pf}

%le3.2 #&#
%
\begin{lem}\label{zz}
For $0\le t<s\le T$, we have
\[
\int_{\bR^n}\sup_{r\in[t,s]}G_{r,t}(x-y)
\llvert x-y \rrvert ^{2\alpha} \,dy<+\infty.
\]
\end{lem}

\begin{pf}
First, consider the following function $\rho$:
\[
\rho(t,r):=t^{-n/2}\exp{ \biggl(-\frac{c}{t}r^2
\biggr)}, \qquad(t, r)\in[0,T]\times\bR.
\]
We have
\[
\partial_t\rho(t,r)= \bigl(cr^2-\tfrac{1}{2}nt
\bigr)t^{-2}\rho(t,r). %
\]
Therefore, the function $\rho(\cdot,r)$ increases on $[0, T]$ for any
fixed $r$ such that $r^2>M^2:={nT\over2c}$, and we have
\[
\sup_{t\in[0, T]} \rho(t,r) \leq\rho(T,r).
\]

In view of estimate (\ref{poten2}) to the heat potential, we have
\begin{eqnarray*}
&&\int_{\bR^n}\llvert x-y \rrvert ^{2\alpha}\sup
_{r\in[t,s]}G_{r,t}(x-y) \,dy
\\
&&\qquad =\int_{\llvert  x-y \rrvert  \leq M}\llvert x-y \rrvert ^{2\alpha}\sup
_{r\in[t,s]}G_{r,t}(x-y) \,dy
\\
&&\quad\qquad{}  +\int_{\llvert  x-y \rrvert  > M}\llvert x-y \rrvert ^{2\alpha}\sup
_{r\in[t,s]}G_{r,t}(x-y) \,dy
\\
&&\qquad \leq C\int_{\llvert  z \rrvert  \leq M}\llvert z \rrvert ^{2\alpha}\sup
_{r\in[t,s]}\rho \bigl(r-t,\llvert z \rrvert \bigr) \,dz
\\
&&\quad\qquad{}  +C\int_{\llvert  z \rrvert  >M}\llvert z \rrvert ^{2\alpha}\sup
_{r\in[t,s]}\rho\bigl(r-t,\llvert z \rrvert \bigr) \,dz
\\
&&\qquad =I_1+C\int_{\llvert  z \rrvert  >M}\llvert z \rrvert
^{2\alpha}\rho\bigl(T,\llvert z \rrvert \bigr) \,dz.
\end{eqnarray*}
Since the second integral is easily verified to be finite, it remains
to show $I_1<\infty$.
Noting that $\rho(t,\llvert  z \rrvert  )$ is maximized at $t=\frac{2}{n}c\llvert  z \rrvert  ^2$ over
$[0, T]$ for any fixed $z$ such that $\llvert  z \rrvert  \leq M$, we have
\[
\sup_{r\in[t,s]}\rho\bigl(r-t, \llvert z \rrvert \bigr)\leq\rho
\biggl(\frac{2}{n}c\llvert z \rrvert ^2, \llvert z \rrvert
\biggr)\leq C\llvert z \rrvert ^{-n}
\]
and
\[
I_1\leq C\int_{\llvert  z \rrvert  \leq M}\llvert z \rrvert
^{-n+2\alpha} \,dz <+\infty.
\]
The proof is then complete.
\end{pf}

In Lemma \ref{repreu}, the expression of $u$ still depends on $v$,
which is unknown. Next, we construct an explicit expression of $(u,v)$
only in terms of the terminal term $\Phi$ and the free term $f$.

Let $\Phi\in C^{1+\alpha}(\bR^n, L^2(\Omega))$ and $f\in C^{\alpha
}(\bR^n, \sL^2_{\bF}(0,T))$. Consider the two family BSDEs~(\ref
{bsde01}) and (\ref{bsde02}): for any $x\in\bR^n$ and almost all
$\tau\in[0,T]$,
their solutions are denoted by $(\varphi(\cdot,x),\psi(\cdot,x))$
and $(Y(\cdot;\tau,x),g(\cdot;\tau,x))$, respectively. From the
theory of BSDEs, we have
\[
(\varphi, \psi)\in C^{1+\alpha}\bigl(\bR^n,
\sS^2_{\bF}[0,T]\bigr)\times C^{1+\alpha}\bigl(
\bR^n, \sL^2_{\bF}\bigl(0,T;\bR^d
\bigr)\bigr),
\]
and there is $C=C(\alpha,n,d)$ such that
%e3.24 #&#
%e3.25 #&#
%e3.26 #&#
%e3.27 #&#
%e3.28 #&#
%
\begin{eqnarray}\label{bsdeest1}
&& \llVert \varphi \rrVert _{1+\alpha,\sS^2}+\llVert \psi \rrVert
_{1+\alpha,\sL^2} \leq C \llVert \Phi \rrVert _{1+\alpha,L^2},
\\
&& \sup_{x}E \int_0^T
\sup_{t\leq r}\bigl\llvert Y(t;r,x)\bigr\rrvert ^2 \,dr\nonumber
\\
&&\quad\qquad{}
+\sup_{x\neq\bar{x}}\frac{E \int_0^T \sup_{t\leq
r}\llvert  Y(t;r,x)-Y(t;r,\bar{x})\rrvert  ^2\,dr}{\llvert  x-\bar{x}\rrvert  ^{2\alpha}}\nonumber
\\
\label{bsdeest2}  &&\quad\qquad {} +\sup_{x}E \int_0^T  \int_0^r\bigl\llvert g(t;r,x)\bigr\rrvert
^2\,dt \,dr
\\
&&\quad\qquad{} +\sup_{x\neq\bar{x}}\frac{E \int_0^T  \int_0^r\llvert  g(t;r,x)-g(t;r,\bar{x})\rrvert  ^2\,dt\,dr}{\llvert  x-\bar{x}\rrvert  ^{2\alpha}}\nonumber
\\
&&\qquad \leq C \llVert f \rrVert ^2_{\alpha,\sL^2}.\nonumber
\end{eqnarray}

We have the following explicit expression of $(u,v)$.
%th3.3 #&#

\begin{teo}\label{representationuv}
Let Assumption \ref{superparab} hold and $(\Phi, f)\in C^{1+\alpha
}(\bR^n, L^2(\Omega))\times C^{\alpha}(\bR^n, \sL^2_{\bF}(0,T))$.
Let $(\varphi,\psi)$ and $(Y,g)$ be solutions of BSDEs~(\ref
{bsde01}) and (\ref{bsde02}), respectively, and $(u,v)\in C^{\alpha
}_{\sS^2}\cap C^{2+\alpha}_{\sL^2}\times C^{\alpha}_{\sL^2}$ solve
BSPDE (\ref{simplelinearbspde}). Then for all $x\in\bR^n$,
%e3.29 #&#
%
\begin{eqnarray}
u(t,x)&=&R^T_t\varphi(t) (x)+\int
_t^TR^s_tY(t;s) (x)
\,ds
\nonumber\\[-8pt]\label{expressu} \\[-8pt]
\eqntext{\forall t\in [0,T], \,dP\mbox{-a.s.},}
\\
%e3.30 #&#
%
v_l(s,x)&=&R^T_s
\psi_l(s) (x)+\int_s^TR^r_sg_l(s;r)
(x) \,dr,
\nonumber\\[-8pt]\label{expressv}  \\[-8pt]
\eqntext{ds\times dP\mbox{-a.e., a.s., }l=1,\ldots,d,}
\end{eqnarray}
where $R^s_t$ is defined by (\ref{note}).
\end{teo}

\begin{pf}
In view of Lemma~\ref{repreu} and the definition~(\ref{nBM}), we see
that for all $(t,x)\in[0,T]\times\bR^n$,
\begin{eqnarray*}
u(t,x)&=&R^T_t\Phi(x)+\int_t^TR^s_tf(s)
(x) \,ds -\int_t^TR^s_tv_l(s)
(x) \,d\widetilde{W}{}^l_s, \qquad P\mbox{-a.s.},
\\
\varphi(t;x)&=&\Phi(x)-\int_t^T
\psi_l(s;x) \,d\widetilde{W}{}^l_s, \qquad P
\mbox{-a.s.},
\end{eqnarray*}
and for almost all $\tau\in[0,T]$ and any $s\leq\tau$,
\[
Y(s;\tau,x)=f(\tau,x)-\int_s^\tau
g_l(r;\tau,x) \,d\widetilde {W}{}^l_r, \qquad P
\mbox{-a.s.}
\]
In view of the stochastic Fubini theorem and semi-group property of
$G_{s,t}$, we have
%e3.31 #&#
%e3.32 #&#
%e3.33 #&#
%e3.34 #&#
%e3.35 #&#
%e3.36 #&#
%
\begin{eqnarray}
&& R^T_t\Phi(x)+\int_t^TR^s_tf(s)
(x) \,ds\nonumber
\\
&&\qquad = R^T_t\varphi(t) (x)+\int_t^T
R^s_tY(t;s) (x) \,ds\nonumber
\\
&&\quad\qquad{}  +\int_t^T R^T_t
\psi_l(s) (x) \,d\widetilde{W}{}^l_s +\int
_t^T R^s_t\int _t^s g_l(r;s) (x) \,d
\widetilde{W}{}^l_r \,ds
\\
&&\qquad = R^T_t\varphi(t) (x)+\int_t^TR^s_tY(t;s)
(x) \,ds\nonumber
\\
&&\quad\qquad{}  +\int_t^TR^s_t
\biggl(R^T_s\psi_l(s)+\int
_s^TR^r_sg_l(s;r)
\,dr \biggr) (x) \,d\widetilde{W}{}^l_s.\nonumber
\end{eqnarray}
Therefore,
%e3.37 #&#
%e3.38 #&#
%e3.39 #&#
%e3.40 #&#
%
\begin{eqnarray}
\label{exp1} && u(t,x)+\int_t^TR^s_tv_l(s)
(x) \,d\widetilde{W}{}^l_s\nonumber
\\
&&\qquad = R^T_t\varphi(t) (x)+\int_t^TR^s_tY(t;s)
(x) \,ds
\\
&&\quad\qquad{}  +\int_t^TR^s_t
\biggl(R^T_s\psi_l(s)+\int
_s^TR^r_sg_l(s;r)
\,dr \biggr) (x) \,d\widetilde{W}{}^l_s.\nonumber
\end{eqnarray}
In view of (\ref{bsdeest1}) and (\ref{bsdeest2}), we have for each
$x\in\mathbb{R}^n$,
\[
\int_t^T\biggl\llvert R^s_t
\biggl(R^T_s\psi(s)+\int_s^TR^r_sg(s;r)
\,dr \biggr) (x)-R^s_tv(s) (x)\biggr\rrvert
^2 \,ds<+\infty,\qquad P\mbox{-a.s.}
\]
Taking on both sides of (\ref{exp1}) the expectation with respect to
the new probability~$Q$ [see~(\ref{nProb}) for the definition]
conditioned on $\sF_t$, we have almost surely
\[
u(t,x)=R^T_t\varphi(t) (x)+\int_t^TR^s_tY(t;s)
(x) \,ds\qquad\forall x \in\mathbb{R}^n; %
\]
and
\[
\int_t^TR^s_t
\biggl(R^T_s\psi_l(s)+\int
_s^TR^r_sg_l(s;r)
\,dr-v_l(s) \biggr) (x) \,d\widetilde{W}{}^l_s=0
\]
for any $(t,x)\in[0,T]\times\mathbb{R}^n$, which implies the following:
%e3.41 #&#
%
\begin{equation}
E \biggl[\int_t^T\biggl\llvert
R^s_t \biggl(R^T_s\psi(s)+\int
_s^TR^r_sg(s;r)
\,dr-v(s) \biggr) (x)\biggr\rrvert ^2 \,ds \biggr]=0
\end{equation}
for any $(t,x)\in[0,T]\times\mathbb{R}^n$.
Then, for all $l=1,\ldots,d$, we have almost surely
\[
V_l(t,x):=\int_t^TR^s_t
\biggl(R^T_s\psi_l(s)+\int
_s^TR^r_sg_l(s;r)
\,dr-v_l(s) \biggr) (x) \,ds=0
\]
for any $(t,x)\in[0,T]\times\mathbb{R}^n$,
which almost surely solves a deterministic PDE and, therefore, the
nonhomogeneous term (the sum in the bigger pair of parentheses in the
last equality) of this PDE is equal to zero. Consequently, we have for
each $x\in\mathbb{R}^n$,
\[
v_l(s,x)=R^T_s\psi_l(s) (x)+
\int_s^TR^r_sg_l(s;r)
(x) \,dr, \qquad ds\times dP\mbox{-a.e., a.s.}
\]
The proof is complete.
\end{pf}

%re3.1 #&#
%
\begin{rmk}\label{derivative}
Let Assumption \ref{superparab} hold and $(\Phi, f)\in C^{1+\alpha
}(\bR^n, L^2(\Omega))\times C^{\alpha}(\bR^n, \sL^2_{\bF}(0,T))$.
Let $(\varphi,\psi)$ and $(Y,g)$ be solutions of BSDEs~(\ref
{bsde01}) and (\ref{bsde02}), respectively. Then, for all $x\in\bR
^n$, $R^T_t\varphi(t)(x)$ and $\int_t^T R^s_tY(t;s)(x)\,ds$ are twice
continuously differentiable in $x$ as $\sL^2_{\bF}(0,T)$-valued
functionals. Moreover, we have for all $(t,x)\in[0,T)\times\bR^n$,
\begin{eqnarray*}
\partial_iR^T_t\varphi(t) (x)&=&\int
_{\bR^n}G_{T,t}(x-y)\,\partial _i
\varphi(t,y) \,dy,\qquad P\mbox{-a.s.},
\\
\partial^2_{ij}R^T_t\varphi(t)
(x)&=&\int_{\bR^n}\partial _iG_{T,t}(x-y)
\bigl[\partial_j\varphi(t,y)-\partial_j\varphi (t,x)
\bigr] \,dy,\qquad P\mbox{-a.s.}
\end{eqnarray*}
and for all $x\in\bR^n$, $dt\times dP$-a.e., a.s.,
\begin{eqnarray*}
&& \partial_i\int_t^T
R^s_tY(t;s) (x)\,ds =\int_t^T  \int_{\bR^n}\partial_iG_{s,t}(x-y)Y(t;s,y)
\,dy \,ds,
\\
&& \partial^2_{ij}\int_t^T
R^s_tY(t;s) (x)\,ds
\\
&&\qquad =\int_t^T  \int_{\bR^n}\partial^2_{ij}G_{s,t}(x-y)
\bigl[Y(t;s,y)-Y(t;s,x) \bigr] \,dy \,ds.
\end{eqnarray*}
\end{rmk}

%s3.2 #&#
\subsection{H\"{o}lder estimates}\label{sec3.2}
Using the explicit expression of $(u,v)$ in Theorem \ref{representationuv}, we shall derive H\"{o}lder estimates for $(u,v)$.
%in relevance to $\Phi\in C^{1+\alpha}(\bR^n, L^2(\Omega))$ and $f\in
%C^{\alpha}(\bR^n, \sL^2_{\bF}(0,T))$.

%le3.4 #&#
%
\begin{lem}\label{estvarphiY}
Let Assumption \ref{superparab} be satisfied and suppose that
\[
(\Phi, f)\in C^{1+\alpha}\bigl(\bR^n, L^2(\Omega)
\bigr)\times C^{\alpha}\bigl(\bR ^n, \sL^2_{\bF}(0,T)
\bigr).
\]
If $(u,v)\in C^{\alpha}_{\sS^2}\cap C^{2+\alpha}_{\sL^2}\times
C^{\alpha}_{\sL^2}$ solves BSPDE (\ref{simplelinearbspde}), then
we have
\[
\llVert u \rrVert _{2+\alpha,\sL^2}\leq C \bigl(\llVert \Phi \rrVert
_{1+\alpha,L^2}+\llVert f \rrVert _{\alpha,\sL^2}\bigr),
\]
where $C=C(\lambda,\Lambda,\alpha,n,d,T)$.
\end{lem}

\begin{pf}
In view of (\ref{expressu}) and (\ref{expressv}), we need to prove
\begin{eqnarray*}
\bigl\llVert R^T_{\cdot}\varphi(\cdot)\bigr\rrVert
_{2+\alpha,\sL^2}&\leq& C \llVert \Phi \rrVert _{1+\alpha,L^2},
\\
\biggl\llVert \int_{\cdot}^TR^s_{\cdot}Y(
\cdot;s) \,ds\biggr\rrVert _{2+\alpha,\sL^2}&\leq& C \llVert f \rrVert
_{\alpha,\sL^2}.
\end{eqnarray*}
It is sufficient to prove the second inequality, and the first one can
be proved in a similar way.

For $\gamma\in\Gamma$ such that $\llvert  \gamma\rrvert  \leq1$, in view of (\ref
{poten2}), (\ref{poten4}), (\ref{bsdeest2}) and Remark \ref
{derivative}, we have
\begin{eqnarray*}
&& E\int_0^T\biggl\llvert D^{\gamma}\int
_t^TR_{s,t}Y(t;s) (x) \,ds\biggr\rrvert
^2 \,dt
\\
&&\qquad = E\int_0^T\biggl\llvert \int
_t^T  \int_{\bR^n}D^{\gamma
}G_{s,t}(x-y)Y(t;s,y)
\,dy \,ds\biggr\rrvert ^2 \,dt
\\
&&\qquad \leq E\int_0^T  \int_t^T  \int_{\bR^n}\bigl\llvert D^{\gamma
}G_{s,t}(x-y)
\bigr\rrvert \bigl\llvert Y(t;s,y)\bigr\rrvert ^2 \,dy \,ds
\\
&&\quad\qquad{}  \times\int_t^T  \int_{\bR^n}
\bigl\llvert D^{\gamma}G_{s,t}(x-y)\bigr\rrvert \,dy \,ds \,dt
\end{eqnarray*}
and, therefore,
\begin{eqnarray*}
&& E\int_0^T\biggl\llvert D^{\gamma}\int
_t^TR_{s,t}Y(t;s) (x) \,ds\biggr\rrvert
^2 \,dt
\\
&&\qquad \leq\int_{\bR^n}E\int_0^T
\sup_{t\leq s} \bigl\llvert Y(t;s,y)\bigr\rrvert ^2
\int_0^s \bigl\llvert D^{\gamma}G_{s,t}(x-y)
\bigr\rrvert \,dt \,ds \,dy
\\
&&\quad\qquad{} \times\sup_{\tau}\int_\tau^T
\int_{\bR^n} \bigl\llvert D^{\gamma
}G_{s,\tau}(x-y)
\bigr\rrvert \,dy \,ds
\\
&&\qquad \leq C \sup_yE\int_0^T
\sup_{t\leq s}\bigl\llvert Y(t;s,y)\bigr\rrvert ^2\,ds
\cdot\int_{\bR^n} \sup_s\int
_0^s \bigl\llvert D^{\gamma}G_{s,t}(x-y)
\bigr\rrvert\, dt\,dy
\\
&&\quad\qquad{}  \times\sup_{\tau}\int_\tau^T
\int_{\bR^n} \bigl\llvert D^{\gamma
}G_{s,\tau}(x-y)
\bigr\rrvert \,dy \,ds
\\
&&\qquad \leq C \sup_yE\int_0^T
\sup_{t\leq s}\bigl\llvert Y(t;s,y)\bigr\rrvert ^2 \,ds
\cdot \biggl\llvert \int_{\bR^n} \sup_{\tau\leq s}
\int_\tau^s \bigl\llvert D^{\gamma
}G_{s,t}(x-y)
\bigr\rrvert \,dt \,dy\biggr\rrvert ^2
\\
&&\qquad \leq C\llVert f \rrVert _{0,\sL^2}^2\biggl\llvert \int
_0^T  \int_{\bR^n}
t^{-({n+\llvert  \gamma\rrvert  })/{2}}\exp{ \biggl(-c\frac{\llvert  x-y \rrvert  ^2}{t} \biggr)} \,dy \,dt\biggr\rrvert
^2
\\
&&\qquad \leq C \llVert f \rrVert _{0,\sL^2}^2.
\end{eqnarray*}
That is,
%e3.42 #&#
%
\begin{equation}
\label{estintf1} \biggl\llVert \int_{\cdot}^TR^s_{\cdot}Y(
\cdot;s) \,ds\biggr\rrVert _{1,\sL
^2}\leq C \llVert f \rrVert
_{0,\sL^2}.
\end{equation}

For $\llvert  \gamma\rrvert  =2$, in view of (\ref{poten2}), (\ref{poten4}), (\ref
{bsdeest2}) and Remark \ref{derivative}, we have
\begin{eqnarray*}
\hspace*{-5pt}&& E\int_0^T\biggl\llvert D^{\gamma}\int
_t^TR^s_tY(t;s) \,ds
\biggr\rrvert ^2 \,dt
\\
\hspace*{-5pt}&&\qquad = C E\int_0^T\biggl\llvert \int
_t^T \! \int_{\bR^n}D^{\gamma
}G_{s,t}(x-y)
\llvert x-y \rrvert ^{\alpha} \frac{\llvert  Y(t;s,y)-Y(t;s,x)\rrvert  }{\llvert  x-y \rrvert  ^{\alpha}} \,dy \,ds\biggr\rrvert
^2 \,dt
\\
\hspace*{-5pt}&&\qquad \leq C \sup_yE\int_0^T
\sup_{t\leq s}\frac
{\llvert  Y(t;s,y)-Y(t;s,x)\rrvert  ^2}{\llvert  x-y \rrvert  ^{2\alpha}}\,ds
\\
\hspace*{-5pt}&&\quad\qquad{}\times\biggl\llvert \int_{\bR^n}\sup_{\tau\leq s}
\int_\tau^s \bigl\llvert D^{\gamma}G_{s,t}(x
- y)\bigr\rrvert \llvert x - y\rrvert ^{\alpha}\,dt \,dy\biggr\rrvert
^2
\\
\hspace*{-5pt}&&\qquad \leq C [f]_{\alpha,\sL^2}^2\biggl\llvert \int_0^T  \int_{\bR
^n}t^{-({n+\llvert  \gamma\rrvert  })/{2}}\exp{ \biggl(-c\frac{\llvert  x-y \rrvert  ^2}{t}
\biggr)} \llvert x-y \rrvert ^{\alpha}\,dy \,dt\biggr\rrvert ^2
\\
\hspace*{-5pt}&&\qquad \leq C [f]_{\alpha,\sL^2}^2.
\end{eqnarray*}
Thus,
%e3.43 #&#
%
\begin{equation}
\label{estintf2} \biggl[\int_{\cdot}^TR^s_{\cdot}Y(t;s)
\,ds \biggr]_{2,\sL^2}\leq C [f]_{\alpha,\sL^2}.
\end{equation}

Define $\upsilon:=2\llvert  x-\bar{x}\rrvert  $ for $x\neq\bar{x}$. By Remark \ref
{derivative}, we have for $\llvert  \gamma\rrvert  =2$,
\begin{eqnarray*}
&& D^{\gamma}\int_t^TR^s_tY(t;s)
(x) \,ds-D^{\gamma}\int_t^TR^s_tY(t;s)
(\bar{x}) \,ds
\\
&&\qquad = \int_t^T  \int_{\bR^n}D^{\gamma}G_{s,t}(x-y)
\bigl[Y(t;s,y)-Y(t;s,x) \bigr] \,dy \,ds
\\
&&\quad\qquad{} -\int_t^T  \int_{\bR^n}D^{\gamma}G_{s,t}(
\bar {x}-y) \bigl[Y(t;s,y)-Y(t;s,\bar{x}) \bigr] \,dy \,ds
\\
&&\qquad = \sum_{i=1}^4 I_i(t,x,
\bar{x})
\end{eqnarray*}
with
%e3.44 #&#
%e3.45 #&#
%e3.46 #&#
%e3.47 #&#
%
\begin{eqnarray}
\qquad I_1&:= &\int_t^T  \int
_{B_{\upsilon}(x)}D^{\gamma
}G_{s,t}(x-y)
\bigl[Y(t;s,y)-Y(t;s,x) \bigr] \,dy \,ds,\nonumber
\\[-2pt]
I_2&:= &-\int_t^T  \int
_{B_{\upsilon}(x)}D^{\gamma
}G_{s,t}(\bar{x}-y)
\bigl[Y(t;s,y)-Y(t;s,\bar{x}) \bigr] \,dy \,ds,\nonumber
\\[-2pt]
I_3&:= &-\int_t^T  \int
_{\llvert  y-x \rrvert  >\upsilon}D^{\gamma
}G_{s,t}(x-y)
\bigl[Y(t;s,x)-Y(t;s,\bar{x}) \bigr] \,dy \,ds,
\\[-2pt]
I_4&:=&\int_t^T  \int
_{\llvert  y-x \rrvert  >\upsilon} \bigl[D^{\gamma}G_{s,t}(x-y)-D^{\gamma}G_{s,t}(
\bar{x}-y) \bigr]\nonumber
\\[-2pt]
&&\hspace*{54pt}{}\times \bigl[Y(t;s,y)-Y(t;s,\bar{x}) \bigr] \,dy\,ds.\nonumber
\end{eqnarray}
Next, we estimate $I_i(t,x,\bar{x})$ for $i=1,2,3,4$. In view of (\ref
{bsdeest2}), Lemma \ref{estpoten3}, and Remark \ref{derivative},
we have
\begin{eqnarray*}
&& E\int_0^T\bigl\llvert I_1(t,x,
\bar{x})\bigr\rrvert ^2 \,dt
\\[-2pt]
&&\qquad = E\int_0^T\biggl\llvert \int
_t^T  \int_{B_{\upsilon}(x)}\bigl\llvert
D^{\gamma
}G_{s,t}(x-y)\bigr\rrvert \llvert x-y \rrvert
^{\alpha}
\\[-2pt]
&&\hspace*{106pt}{}\times \frac{\llvert  Y(t;s,y)-Y(t;s,x)\rrvert  }{\llvert  x-y \rrvert  ^{\alpha}} \,dy \,ds\biggr\rrvert ^2 \,dt
\\[-2pt]
&&\qquad \leq C\sup_y E\int_0^T
\sup_{t\leq s}\frac
{\llvert  Y(t;s,y)-Y(t;s,x)\rrvert  ^2}{\llvert  x-y \rrvert  ^{2\alpha}}\,ds
\\[-2pt]
&&\qquad\quad{} \times\biggl\llvert \int_{B_{\upsilon}(x)} \sup_{\tau\leq s}
\int_\tau^s \bigl\llvert D^{\gamma}G_{s,t}(x-y)
\bigr\rrvert \llvert x-y \rrvert ^{\alpha}\,dt \,dy\biggr\rrvert ^2
\\[-2pt]
&&\qquad \leq C[f]_{\alpha,\sL^2}^2\llvert x-\bar{x}\rrvert ^{2\alpha}.
\end{eqnarray*}
In the same way, we have
\[
E\int_0^T\bigl\llvert I_2(t,x,
\bar{x})\bigr\rrvert ^2 \,dt\leq C[f]_{\alpha,\sL
^2}^2
\llvert x-\bar{x}\rrvert ^{2\alpha}.
\]
From Lemma \ref{estpoten1}, we have
\begin{eqnarray*}
&& E\int_0^T\bigl\llvert I_3(t,x,
\bar{x})\bigr\rrvert ^2 \,dt
\\[-2pt]
&&\qquad \leq E\int_0^T\biggl\llvert \int
_t^T \bigl\llvert Y(t;s,x)-Y(t;s,\bar{x})\bigr
\rrvert \biggl\llvert \int_{\llvert  y-x \rrvert  >\upsilon} D^{\gamma}G_{s,t}(x-y)
\,dy\biggr\rrvert \,ds\biggr\rrvert ^2 \,dt
\\[-2pt]
&&\qquad \leq E\int_0^T  \int_t^T
\bigl\llvert Y(t;s,x)-Y(t;s,\bar{x})\bigr\rrvert ^2\biggl\llvert \int
_{\llvert  y-x \rrvert  >\upsilon} D^{\gamma}G_{s,t}(x-y) \,dy\biggr\rrvert
\,ds \,dt
\\[-2pt]
&&\quad\qquad {}\times \sup_t\int_t^T
\biggl\llvert \int_{\llvert  y-x \rrvert  >\upsilon}D^{\gamma}G_{s,t}(\bar
{x}-y) \,dy\biggr\rrvert \,ds
\\[-2pt]
&&\qquad \leq CE\int_0^T \sup_{t\leq s}
\bigl\llvert Y(t;s,x)-Y(t;s,\bar{x})\bigr\rrvert ^2\,ds
\\[-2pt]
&&\qquad\quad{} \times\sup_s\int_0^s
\biggl\llvert \int_{\llvert  y-x \rrvert  >\upsilon}D^{\gamma
}G_{s,t}(x-y)
\,dy\biggr\rrvert \,dt
\\[-2pt]
&&\qquad \leq C[f]_{\alpha,\sL^2}^2\llvert x-\bar{x}
\rrvert ^{2\alpha}.
\end{eqnarray*}
For $I_4(t,x,\bar{x})$, in view of (\ref{poten2}), (\ref
{bsdeest2}) and Lemma \ref{estpoten4}, we have
\begin{eqnarray*}
&&E\int_0^T\bigl\llvert I_4(t,x,
\bar{x})\bigr\rrvert ^2 \,dt
\\
&&\qquad \leq CE\int_0^T\biggl\llvert \int
_t^T  \int_{\llvert  y-x \rrvert  >\upsilon}
\bigl\llvert D^{\gamma}G_{s,t}(x-y)-D^{\gamma}G_{s,t}(
\bar{x}-y)\bigr\rrvert
\\
&&\hspace*{138pt}{}\times
\bigl\llvert Y(t;s,y)-Y(t;s,\bar{x})\bigr\rrvert
\,dy\,ds\biggr\rrvert ^2\,dt
\\
&&\qquad \leq C \sup_yE\int_0^T
\sup_{t\leq s}\frac
{\llvert  Y(t;s,y)-Y(t;s,\bar{x})\rrvert  ^2}{\llvert  \bar{x}-y\rrvert  ^{2\alpha}}\,ds
\\
&&\quad\qquad{}\times \biggl[\int_{\llvert  y-x \rrvert  >\upsilon}\sup_{\tau\leq s}
\int_\tau^s \bigl\llvert D^{\gamma}G_{s,t}(x-y)-D^{\gamma}G_{s,t}(
\bar{x}-y)\bigr\rrvert \llvert \bar{x}-y\rrvert ^{\alpha}\,dt \,dy
\biggr]^2
\\
&&\qquad \leq C[f]_{\alpha,\sL^2}^2\llvert x-\bar{x}\rrvert
^{2\alpha}.
\end{eqnarray*}
In summary, we have
%e3.48 #&#
%
\begin{equation}
\label{estintf3} \sum_{i=1}^4E\int
_0^T\bigl\llvert I_i(t,x,\bar{x})
\bigr\rrvert ^2 \,dt\leq C [f]_{\alpha,\sL^2}^2\llvert x-
\bar{x}\rrvert ^{2\alpha}.
\end{equation}

Combining (\ref{estintf1}), (\ref{estintf2}) and (\ref
{estintf3}), we have
\[
\biggl\llVert \int_{\cdot}^TR^s_{\cdot}Y(
\cdot;s) \,ds\biggr\rrVert _{2+\alpha,\sL^2}\leq C \llVert f \rrVert
_{\alpha,\sL^2}.
\]\upqed
\end{pf}

%le3.5 #&#
%
\begin{lem}\label{estpsig}
Let Assumption \ref{superparab} be satisfied and suppose that
\[
(\Phi,f)\in C^{1+\alpha}\bigl(\bR^n, L^2(\Omega)
\bigr)\times C^{\alpha}\bigl(\bR ^n, \sL^2_{\bF}(0,T)
\bigr).
\]
If $(u,v)\in C^{\alpha}_{\sS^2}\cap C^{2+\alpha}_{\sL^2}\times
C^{\alpha}_{\sL^2}$ solves BSPDE (\ref{simplelinearbspde}), then
we have
\[
\llVert v \rrVert _{\alpha,\sL^2}\leq C \bigl(\llVert \Phi \rrVert
_{1+\alpha,L^2}+\llVert f \rrVert _{\alpha,\sL^2}\bigr),
\]
where $C=C(\lambda,\Lambda,\alpha,n,d,T)$.
\end{lem}

\begin{pf}
In view of (\ref{expressu}) and (\ref{expressv}), we need to prove
\begin{eqnarray*}
\bigl\llVert R^T_{\cdot}\psi(\cdot)\bigr\rrVert
_{\alpha,\sL^2}&\leq& C \llVert \Phi \rrVert _{1+\alpha,L^2},
\\
\biggl\llVert \int_{\cdot}^TR^s_{\cdot}g(
\cdot;s) \,ds\biggr\rrVert _{\alpha,\sL^2}&\leq& C \llVert f \rrVert
_{\alpha,\sL^2}.
\end{eqnarray*}
It is sufficient to prove the second inequality, and the first one can
be proved in a similar way.

For all $x\in\bR^n$, and almost all $r\in[0,T]$,
\begin{eqnarray}
R^r_tY(t;r) (x)&=&f(r,x)+\int_t^r
\sigma(s)R^r_sg(s;r) (x) \,ds-\int_t^rR^r_sg(s;r)
(x) \,dW_s\nonumber
\\
&&{}+\int_t^ra_{ij}(s)\int
_{\bR^n}\partial ^2_{ij}G_{r,s}(x-y)
\bigl[Y(s;r,y)-Y(s;r,x) \bigr] \,ds\nonumber
\\
\eqntext{\forall t\leq r.}
\end{eqnarray}
From the theory of BSDEs, we have
\begin{eqnarray*}
\hspace*{-2pt}&&E \biggl[\int_0^r\bigl\llvert
R^r_sg(s;r) (x)\bigr\rrvert ^2 \,ds \biggr]
\\
\hspace*{-2pt}&&\qquad \leq C E \biggl[\bigl\llvert f(r,x)\bigr\rrvert ^2
\\
\hspace*{-2pt}&&\hspace*{21pt}\qquad\quad{}  +\int_0^r\biggl\llvert a_{ij}(s)
\int_{\bR^n}\partial^2_{ij}G_{r,s}(x-y)
\bigl[Y(s;r,y)-Y(s;r,x) \bigr]\,dy\biggr\rrvert ^2\,ds \biggr]
\\
\hspace*{-2pt}&&\qquad \leq C E \biggl[\bigl\llvert f(r,x)\bigr\rrvert ^2+\int
_0^r\biggl\llvert \int_{\bR^n}
\partial^2_{ij}G_{r,s}(x-y) \bigl[Y(s;r,y)-Y(s;r,x)
\bigr] \,dy\biggr\rrvert ^2\,ds \biggr].
\end{eqnarray*}
Integrating both sides on $[0,T]$, we have
\begin{eqnarray*}
&& E \biggl[\int_0^T  \int_s^T
\bigl\llvert R^r_sg(s;r) (x)\bigr\rrvert ^2
\,dr \,ds \biggr]
\\
&&\qquad =E \biggl[\int_0^T  \int
_0^r\bigl\llvert R^r_sg(s;r)
(x)\bigr\rrvert ^2 \,ds \,dr \biggr]
\\
&&\qquad \leq C E \biggl[\int_0^T\bigl\llvert f(r,x)
\bigr\rrvert ^2 \,dr
\\
&&\hspace*{21pt}\quad\qquad{} +\int_0^T  \int_0^r
\biggl\llvert \int_{\bR^n}\partial^2_{ij}G_{r,s}(x-y)
\bigl[Y(s;r,y)-Y(s;r,x) \bigr] \,dy\biggr\rrvert ^2 \,ds \,dr \biggr]
\\
&&\qquad \leq C E \biggl[\int_0^T\bigl\llvert f(r,x)
\bigr\rrvert ^2 \,dr
\\
&&\hspace*{20pt}\quad\qquad{}+\int_0^T  \int_s^T
\biggl\llvert \int_{\bR^n}\partial^2_{ij}G_{r,s}(x-y)
\bigl[Y(s;r,y)-Y(s;r,x) \bigr] \,dy\biggr\rrvert ^2 \,dr \,ds \biggr].
\end{eqnarray*}
Similarly,
\begin{eqnarray*}
&&E \biggl[\int_0^T  \int_s^T
\bigl\llvert R^r_sg(s;r) (x)-R^r_sg(s;r)
(\bar {x})\bigr\rrvert ^2 \,dr \,ds \biggr]
\\
&&\qquad \leq C E \biggl[\int_0^T\bigl\llvert
f(r,x)-f(r,\bar{x})\bigr\rrvert ^2 \,dr \biggr]
\\
&&\quad\qquad{} +CE \biggl[\int_0^T  \int
_s^T\biggl\llvert \int_{\bR^n}
\bigl(\partial^2_{ij}G_{r,s}(x-y)
\bigl[Y(s;r,y)-Y(s;r,x) \bigr]
\\
&&\hspace*{123pt}{} -\partial^2_{ij}G_{r,s}(
\bar{x}-y)
\\
&&\hspace*{177pt}{}\times \bigl[Y(s;r,y)-Y(s;r,\bar{x}) \bigr] \bigr) \,dy\biggr\rrvert
^2 \,dr \,ds \biggr].
\end{eqnarray*}
In view of the proof in Lemma \ref{estvarphiY}, we have
\[
\biggl\llVert \int_{\cdot}^TR^s_{\cdot}g(
\cdot;s) \,ds\biggr\rrVert _{\alpha,\sL^2}\leq C \llVert f \rrVert
_{\alpha,\sL^2}.
\]\upqed
\end{pf}

We have the following H\"{o}lder estimate for $(u,v)$.
%th3.6 #&#

\begin{teo}\label{priorest1}
Let Assumption \ref{superparab} be satisfied and
\[
(\Phi, f)\in C^{1+\alpha}\bigl(\bR^n, L^2(\Omega)
\bigr)\times C^{\alpha}\bigl(\bR ^n, \sL^2_{\bF}(0,T)
\bigr).
\]
If $(u,v)$ is a classical solution to BSPDE (\ref{simplelinearbspde}),
then
\[
\llVert u \rrVert _{\alpha,\sS^2}+\llVert u \rrVert _{2+\alpha,\sL^2}+\llVert v
\rrVert _{\alpha,\sL
^2}\leq C \bigl(\llVert \Phi \rrVert _{1+\alpha,L^2}+\llVert
f \rrVert _{\alpha,\sL^2} \bigr),
\]
where $C=C(\lambda,\Lambda,\alpha,n,d,T)$.
\end{teo}

\begin{pf}
From Theorem \ref{representationuv} and the Lemmas \ref
{estvarphiY} and \ref{estpsig}, we have
\[
\llVert u \rrVert _{2+\alpha,\sL^2}+\llVert v \rrVert _{\alpha,\sL^2}\leq C
\bigl(\llVert \Phi \rrVert _{1+\alpha,L^2}+\llVert f \rrVert _{\alpha,\sL^2}
\bigr).
\]
Since $(u,v)$ is the solution of BSPDE (\ref{simplelinearbspde}),
for all $(t,x)\in[0,T]\times\bR^n$, the equality holds almost surely:
\begin{eqnarray*}
u(t,x) &=& \Phi(x)+\int_t^T \bigl[a^{ij}(s)\,
\partial ^2_{ij}u(s,x)+f(s,x)+\sigma(s)v(s,x)\bigr]\,ds
\\[3pt]
&&{} -\int_t^T v(s,x)\,dW_s.
\end{eqnarray*}
For each $x$, it is a BSDE of terminal value $\Phi(x)$ and generator
$a^{ij}(t)\,\partial^2_{ij}u(t,x)+f(t,x)+\sigma
(t)V$. From the theory of BSDEs, we have
\begin{eqnarray*}
E \Bigl[\sup_t\bigl\llvert u(t,x)\bigr\rrvert
^2 \Bigr]&\leq& C E \biggl[\bigl\llvert \Phi(x)\bigr\rrvert
^2+\int_0^T\bigl\llvert
a^{ij}(s)\,\partial^2_{ij}u(s,x)+f(s,x)\bigr\rrvert
^2 \,ds \biggr]
\\[3pt]
&\leq& C E \biggl[\bigl\llvert \Phi(x)\bigr\rrvert ^2+\int
_0^T \bigl(\bigl\llvert \partial
^2_{ij}u(s,x)\bigr\rrvert ^2+\bigl\llvert
f(s,x)\bigr\rrvert ^2 \bigr) \,ds \biggr]
\end{eqnarray*}
and for all $\bar{x}\neq x$,
\begin{eqnarray*}
&&E \Bigl[\sup_t\bigl\llvert u(t,x)-u(t,\bar{x})\bigr
\rrvert ^2 \Bigr]
\\
&&\qquad \leq C E \bigl[\bigl\llvert \Phi(x)-\Phi(\bar{x})\bigr\rrvert ^2
\bigr]
\\
&&\quad\qquad{} +C E \biggl[\int_0^T \bigl(\bigl\llvert
\partial^2_{ij}u(s,x)-\partial ^2_{ij}u(s,
\bar{x})\bigr\rrvert ^2 +\bigl\llvert f(s,x)-f(s,\bar{x})\bigr\rrvert
^2 \bigr) \,ds \biggr].
\end{eqnarray*}
Then
\begin{eqnarray*}
\llVert u \rrVert _{\alpha,\sS^2}&\leq& C \bigl(\llVert \Phi \rrVert
_{\alpha,L^2}+\llVert u \rrVert _{2+\alpha,\sL^2}+\llVert f \rrVert
_{\alpha,\sL^2} \bigr)
\\
&\leq& C \bigl(\llVert \Phi \rrVert _{1+\alpha,L^2}+
\llVert f \rrVert _{\alpha,\sL^2} \bigr).
\end{eqnarray*}
The proof is complete.
\end{pf}

%\begin{rmk}\label{regularv} In the preceding theorem,
% for $f\in C^{1+\alpha}(\bR^n, \sL^2_{\bF}(0,T))$, similar to the
%proof of Lemma \ref{estpsig}, we have
% $$\left\Vert \int_{\cdot}^TR^s_{\cdot}g(\cdot;s) \,ds\Bigr\left\Vert _{1+\alpha,
%\sL^2}\leq C  \left\Vert f \right\Vert _{1+\alpha,\sL^2}.$$
% In view of $\Phi\in C^{1+\alpha}(\bR^n, L^2(\Omega))$, we have $v\in
%C^{1+\alpha}(\bR^n, \sL^2)$ and
% $$\left\Vert v \right\Vert _{1+\alpha,\sL^2}\leq C \Big(\left\Vert \Phi  \right\Vert _{1+\alpha,L^2}+\left\Vert f \right\Vert _{1+
%\alpha,\sL^2}\Big).$$
%\end{rmk}

%s3.3 #&#
\subsection{Existence and uniqueness}\label{sec3.3}

%th3.7 #&#
%
\begin{teo}\label{exist1}
Let Assumption \ref{superparab} be satisfied and
\[
(\Phi, f)\in C^{1+\alpha}\bigl(\bR^n, L^2(\Omega)
\bigr)\times C^{\alpha}\bigl(\bR ^n, \sL^2_{\bF}(0,T)
\bigr).
\]
Let $(\varphi,\psi)$ and $(Y,g)$ be solutions of BSDEs~(\ref
{bsde01}) and (\ref{bsde02}), respectively. Then, the pair $(u,v)$ of
random fields defined by
\begin{eqnarray*}
u(t,x)&=&R^T_t\varphi(t) (x)+\int_t^TR^s_tY(t;s)
(x) \,ds,
\\
v(t,x)&=&R^T_t\psi(t) (x)+\int_t^TR^T_sg(t;s)
(x) \,ds %
\end{eqnarray*}
is the unique classical solution to BSPDE (\ref{simplelinearbspde}).
Moreover, $(u,v)\in(C^\alpha_{\sS^2}\cap C^{2+\alpha}_{\sL
^2})\times C^{\alpha}_{\sL^2}$, and
\[
\llVert u \rrVert _{\alpha,\sS^2}+\llVert u \rrVert _{2+\alpha,\sL^2}+\llVert v
\rrVert _{\alpha,\sL
^2}\leq C \bigl(\llVert \Phi \rrVert _{1+\alpha,L^2}+\llVert
f \rrVert _{\alpha,\sL^2} \bigr),
\]
where $C=C(\lambda,\Lambda,\alpha,n,d,T)$.
\end{teo}

\begin{pf}
In view of Remark \ref{derivative}, for all $(t,x)\in[0,T]\times\bR
^n$, we have
%e3.49 #&#
%e3.50 #&#
%e3.51 #&#
%e3.52 #&#
%e3.53 #&#
%e3.54 #&#
%e3.55 #&#
%
\begin{eqnarray}
\label{existchange1} && \int_t^Ta^{ij}(s)\,
\partial^2_{ij}R^T_s\varphi(s) (x)
\,ds\nonumber
\\[-1pt]
&&\qquad = \int_t^T  \int_{\bR^n}a^{ij}(s)\,
\partial ^2_{ij}G_{T,s}(x-y)\bigl[\varphi(s,y)-
\varphi(s,x)\bigr] \,dy \,ds\nonumber
\\[-1pt]
&&\qquad = \int_{\bR^n}\int_t^T -
\frac{\partial}{\partial s} G_{T,s}(x-y)\bigl[\varphi(s,y)-\varphi(s,x)\bigr] \,dy
\,ds\nonumber
\\[-1pt]
&&\qquad = \int_{\bR^n} G_{T,t}(x-y)\bigl[\varphi(t,y)-
\varphi(t,x)\bigr] \,dy
\\[-1pt]
&&\quad\qquad{} +\int_{\bR^n}\int_t^T
G_{T,s}(x-y)\,d\bigl[\varphi(s,y)-\varphi(s,x)\bigr] \,dy\nonumber
\\[-1pt]
&&\qquad =  R^T_t\varphi(t) (x)-\varphi(t,x)+\int
_t^TR^T_s \psi_l(s) (x) \,d\widetilde{W}{}^l_s\nonumber
\\[-1pt]
&&\quad\qquad{} -\int_t^T\psi_l(s,x) \,d
\widetilde{W}{}^l_s\nonumber
\\[-1pt]
&&\qquad = R^T_t\varphi(t) (x)-\Phi(x)+\int
_t^TR^T_s
\psi_l(s) (x) \,d\widetilde{W}{}^l_s.\nonumber
\end{eqnarray}
Similarly, we have
%e3.56 #&#
%e3.57 #&#
%e3.58 #&#
%e3.59 #&#
%e3.60 #&#
%e3.61 #&#
%e3.62 #&#
%
\begin{eqnarray}
\label{existchange2} &&\int_t^Ta^{ij}(s)\,
\partial^2_{ij}\int_s^TR^r_sY(s;r)
(x) \,dr \,ds\nonumber
\\
&&\qquad = \int_t^T  \int_s^T  \int_{\bR^n} a^{ij}(s)\,\partial ^2_{ij}G_{r,s}(x-y)
\bigl[Y(s;r,y)-Y(s;r,x) \bigr] \,dy \,dr \,ds\nonumber\hspace*{-15pt}
\\
&&\qquad =  \int_t^T  \int_s^T  \int_{\bR^n}-\frac{\partial
}{\partial s} G_{r,s}(x-y)
\bigl[Y(s;r,y)-Y(s;r,x) \bigr] \,dy \,dr \,ds\nonumber
\\
&&\qquad =  \int_t^T  \int_{\bR^n}
\int_t^r-\frac{\partial}{\partial
s}G_{r,s}(x-y) \bigl[Y(s;r,y)-Y(s;r,x) \bigr] \,ds \,dy \,dr
\nonumber\\[-8pt]\\[-8pt]\nonumber
&&\qquad = \int_t^T  \int_{\bR^n}G_{r,t}(x-y)
\bigl[Y(t;r,y)-Y(t;r,x) \bigr] \,dy \,dr\nonumber
\\
&&\quad\qquad{} +\int_t^T  \int_{\bR^n}
\int_t^rG_{r,s}(x-y)\,d_s
\bigl[Y(s;r,y)-Y(s;r,x) \bigr] \,dy \,dr\nonumber
\\
&&\qquad = -\int_t^Tf(r,x) \,dr+\int
_t^TR^r_tY(t;r) (x)\,dr\nonumber
\\
&&\quad\qquad{} +\int_t^T  \int_s^TR^r_sg_l(s;r)
(x) \,dr \,d\widetilde{W}{}^l_s.\nonumber
\end{eqnarray}
In view of (\ref{existchange1}) and (\ref{existchange2}), we have
\begin{eqnarray*}
&&\int_t^Ta^{ij}(s)\,
\partial_{ij}^2u(s,x) \,ds
\\[-1pt]
&&\qquad = \int_t^Ta^{ij}(s)\,
\partial_{ij}^2 \bigl[R^T_s
\varphi(s) (x) \bigr] \,ds +\int_t^Ta^{ij}(s)\,
\partial_{ij}^2 \biggl[\int_s^TR^r_sY(s;r)
(x) \,dr \biggr] \,ds
\\[-1pt]
&&\qquad = R^T_t\varphi(t) (x)-\Phi(x)+\int
_t^TR^T_s
\psi_l(s) (x) \,d\widetilde{W}{}^l_s -\int
_t^Tf(r,x) \,dr
\\[-1pt]
&&\quad\qquad{}+\int_t^TR^r_tY(t;r,y)
\,dr+\int_t^T  \int_s^TR^r_sg_l(s;r)
(x) \,dr \,d\widetilde{W}{}^l_s
\\[-1pt]
&&\qquad = -\Phi(x)-\int_t^Tf(r,x) \,dr+u(t,x)+\int
_t^Tv_l(s,x) \,d\widetilde{W}{}^l_s.
\end{eqnarray*}
Thus, $(u,v)$ solves BSPDE (\ref{simplelinearbspde}), that is,
\begin{eqnarray*}
u(t,x)&=& \Phi(x)+\int_t^T
\bigl[a^{ij}(s)\,\partial _{ij}^2u(s,x)+f(s,x)+
\sigma(s)v(s,x) \bigr] \,ds
\\
&&{}  -\int_t^Tv(s,x) \,dW_s.
\end{eqnarray*}

The desired estimate follows from Theorem \ref{priorest1}. The proof
is complete.
\end{pf}

Moreover, we have the following H\"{o}lder continuity of $u$ in time
$t$. For any $\tau\in[0,T]$, denote by $\llVert  \cdot  \rrVert  _{m+\alpha,\sS
^2,\tau}$ and $\llVert  \cdot  \rrVert  _{m+\alpha,\sL^2,\tau}$ the obvious H\"
{o}lder norms of a process restricted to the time interval $[\tau,T]$.

%pr3.8 #&#
%
\begin{prop}\label{contintime}
Let Assumption \ref{superparab} be satisfied and
\[
(\Phi,f)\in C^{1+\alpha}\bigl(\bR^n, L^2(\Omega)
\bigr)\times C^{\alpha}\bigl(\bR ^n, \sL^2_{\bF}(0,T)
\bigr). %
\]
Let $(u,v)\in C^{\alpha}_{\sS^2}\cap C^{2+\alpha}_{\sL^2}\times
C^{\alpha}_{\sL^2}$ be the\vspace*{1pt} classical solution to BSPDE (\ref
{simplelinearbspde}). Then for any $\tau\in[0,T]$, we have
\[
\bigl\llVert u(\cdot,\cdot)-u(\cdot-\tau,\cdot)\bigr\rrVert _{\alpha,\sL^2,\tau}
\leq C \sqrt{\tau} \bigl(\llVert \Phi \rrVert _{1+\alpha,L^2}+\llVert f \rrVert
_{\alpha,\sL^2} \bigr),
\]
where $C=C(\lambda,\Lambda,\alpha,T,n,d)$.
\end{prop}

\begin{pf}
Since $(u,v)$ satisfies BSPDE (\ref{simplelinearbspde}), we have
\begin{eqnarray*}
&&E\int_{\tau}^T\bigl\llvert u(t,x)-u(t-\tau,x)
\bigr\rrvert ^2 \,dt
\\
&&\qquad \leq CE\int_{\tau}^T\biggl\llvert \int
_{t-\tau}^t \bigl(a^{ij}(s)\, \partial^2_{ij}u(s,x)+f(s,x)+\sigma(s)v(s,x) \bigr) \,ds\biggr
\rrvert ^2 \,dt
\\
&&\quad\qquad{} +CE\int_{\tau}^T\biggl
\llvert \int_{t-\tau}^tv(s,x) \,dW_s
\biggr\rrvert ^2 \,dt
\\
&&\qquad \leq C E\int_0^T  \int_{s\vee\tau}^{T\wedge(s+\tau)}
\bigl(\bigl\llvert \partial^2_{ij}u(s,x)\bigr\rrvert
^2+\bigl\llvert f(s,x)\bigr\rrvert ^2+\bigl\llvert v(s,x)
\bigr\rrvert ^2 \bigr) \,dt \,ds
\\
&&\quad\qquad{}+CE\int_0^T  \int
_{s\vee\tau}^{T\wedge(s+\tau
)}\bigl\llvert v(s,x)\bigr\rrvert
^2 \,dt \,ds
\\
&&\qquad \leq C\tau \bigl([u]^2_{2,\sL^2}+[f]^2_{0,\sL^2}+[v]^2_{0,\sL
^2}
\bigr)
\\
&&\qquad \leq C \tau \bigl(\llVert \Phi \rrVert ^2_{1+\alpha,L^2}+\llVert f
\rrVert ^2_{\alpha,\sL
^2} \bigr).
\end{eqnarray*}
Similarly, for any $x\neq\bar{x}$,
\begin{eqnarray*}
&&E\int_{\tau}^T\bigl\llvert u(t,x)-u(t-\tau,x)-
\bigl[u(t,\bar{x})-u(t-\tau,\bar{x})\bigr]\bigr\rrvert ^2 \,dt
\\
&&\qquad \leq C \tau \bigl(\llVert \Phi \rrVert ^2_{1+\alpha,L^2}+\llVert f
\rrVert ^2_{\alpha,\sL
^2} \bigr)\llvert x-\bar{x}\rrvert
^{2\alpha}.
\end{eqnarray*}
Therefore, we have the desired result.
\end{pf}

%s4 #&#
\section{BSPDEs with space-variable coefficients}\label{sec4}

In this section, using the conventional combinational techniques of the
freezing coefficients method and the parameter continuation argument
well developed in the theory of deterministic PDEs, we extend the a
priori H\"{o}lder estimates as well as the existence and uniqueness
result for BSPDE of the preceding section to the more general BSPDE
(\ref{bspde}).

Consider a smooth function $\varphi\in C^{\infty}_0(\bR^n)$ such that
\[
0\leq\varphi\leq1\quad\mbox{and}\quad \varphi(x)=\cases{ 1, &\quad$\llvert x
\rrvert \leq1$,
\vspace*{3pt}\cr
0, &\quad $ \llvert x\rrvert >2$.}
\]
For any $z\in\bR^n$ and $\theta>0$ fixed, define
\[
\eta^z_{\theta}(x):=\varphi\biggl(\frac{x-z}{\theta}\biggr).
\]
We easily see that for $\gamma\in\Gamma$, there is a constant
$C=C(\gamma,n)$ such that
\[
\bigl[D^{\gamma}\eta^z_{\theta}\bigr]_{0}
\leq C \theta^{-\llvert  \gamma\rrvert  },\qquad \bigl[D^{\gamma}\eta^z_{\theta}
\bigr]_{\alpha}\leq C \theta^{-\llvert  \gamma\rrvert  -\alpha}.
\]

%le4.1 #&#
%
\begin{lem}\label{freezest}
Let $h\in C^{m+\alpha}(\bR^n;\sL^2_{\bF}(0,T;\bR^{\iota}))$ with
$m=0,1,2$. Then there is a positive constant $C(\theta,\alpha)$ such that
\[
\llVert h\rrVert _{m+\alpha,\sL^2}\leq2 \sup_{z\in\bR^n}\bigl\llVert
\eta^z_{\theta}h\bigr\rrVert _{m+\alpha,\sL^2}+C(\theta,
\alpha)\llVert h\rrVert _{0,\sL^2}.
\]
\end{lem}

\begin{pf}
It is sufficient to prove
\[
[h]_{2+\alpha,\sL^2}\leq2 \sup_{z\in\bR^n}\bigl[
\eta^z_{\theta
}h\bigr]_{2+\alpha,\sL^2}+C(\theta,\alpha)\llVert
h\rrVert _{0,\sL^2}.
\]
The proof of the rest is similar.

For any $\theta>0$ fixed, we have
\[
[h]_{2+\alpha,\sL^2}\le I_1+I_2
\]
with
\[
I_1:= \sum_{\llvert  \gamma\rrvert  =2}\sup
_{\llvert  x-\bar{x}\rrvert  <\theta}\frac{E
[\int_0^T\llvert  D^{\gamma}h(t,x)-D^{\gamma}h(t,\bar{x})\rrvert  ^2\,dt ]
^{1/2}}{\llvert  x-\bar{x}\rrvert  ^{\alpha}}
\]
and
\[
I_2:= \sum_{\llvert  \gamma\rrvert  =2}\sup
_{\llvert  x-\bar{x}\rrvert  \geq\theta} \frac{E [\int_0^T\llvert  D^{\gamma}h(t,x)
-D^{\gamma}h(t,\bar{x})\rrvert  ^2\,dt ]^{1/2}}{\llvert  x-\bar
{x}\rrvert  ^{\alpha}}.
\]

For any $x,\bar{x}\in\bR^n$, if $\llvert  x-\bar{x}\rrvert  <\theta$, choose $z=x$,
\begin{eqnarray*}
I_1&\leq &\sum_{\llvert  \gamma\rrvert  =2}\sup
_{\llvert  x-\bar{x}\rrvert  <\theta}\frac{E
[\int_0^T\llvert  D^{\gamma}
 (\eta^x_{\theta}(x)h(t,x) )-D^{\gamma} (\eta^x_{\theta
}(\bar{x})h(t,\bar{x}) )\rrvert  ^2 ]
^{1/2}}{\llvert  x-\bar{x}\rrvert  ^{\alpha}}
\\
&\leq &\sup_{z\in\bR^n}\bigl[\eta^z_{\theta}h
\bigr]_{2+\alpha,\sL^2}.
\end{eqnarray*}
If $\llvert  x-\bar{x}\rrvert  \geq\theta$, using the interpolation inequality in
Lemma \ref{interpolation}, we have
\begin{eqnarray*}
I_2&\leq &\sum_{\llvert  \gamma\rrvert  =2}\sup
_{\llvert  x-\bar{x}\rrvert  \geq\theta}E \biggl[\int_0^T\bigl
\llvert D^{\gamma}h(t,x) -D^{\gamma}h(t,\bar{x})\bigr\rrvert
^2 \,dt \biggr]^{1/2}\theta^{-\alpha}
\\
&\leq &C(\theta,\alpha)[h]_{2,\sL^2} \leq \frac{1}{2}[h]_{2+\alpha,\sL^2}+C(
\theta,\alpha)[h]_{0,\sL^2}.
\end{eqnarray*}
Then
\[
[h]_{2+\alpha,\sL^2}\leq2 \sup_{z\in\bR^n}\bigl[
\eta^z_{\theta
}h\bigr]_{2+\alpha,\sL^2}+C\llVert h\rrVert
_{0,\sL^2}.
\]\upqed
\end{pf}

We have the following a priori H\"{o}lder estimate on the solution
$(u,v)$ to BSPDE~(\ref{bspde}).
%th4.2 #&#

\begin{teo}\label{priorest2}
Let the\vspace*{2pt} Assumptions \ref{superparab2} and \ref{boundedness} be
satisfied and $(\Phi, f)\in C^{1+\alpha}(\bR^n, L^2(\Omega))\times
C^{\alpha}(\bR^n, \sL^2_{\bF}(0,T))$. If $(u,v)\in C^{\alpha}_{\sS
^2}\cap C^{2+\alpha}_{\sL^2}\times C^{\alpha}_{\sL^2}$ solves BSPDE
(\ref{bspde}), we have
\[
\llVert u \rrVert _{\alpha,\sS^2}+\llVert u \rrVert _{2+\alpha,\sL^2}+\llVert v
\rrVert _{\alpha,\sL
^2}\leq C \bigl(\llVert \Phi \rrVert _{1+\alpha,L^2}+\llVert
f \rrVert _{\alpha,\sL^2} \bigr),
\]
where $C=C(\lambda,\Lambda,\alpha,n,d,T)$.
\end{teo}

\begin{pf}
For any $z\in\bR^n$ and $\theta>0$, denote
\begin{eqnarray*}
u^z_{\theta}(t,x)&:=&\eta^z_{\theta}(x)u(t,x),
\qquad v^z_{\theta
}(t,x):=\eta^z_{\theta}(x)v(t,x),
\\
\Phi^z_{\theta}(x)&:=&\eta ^z_{\theta}(x)
\Phi(x),
\end{eqnarray*}
and
\begin{eqnarray*}
f^z_{\theta}(t,x)&:=&\bigl[a^{ij}(t,x)-a^{ij}(t,z)
\bigr]\,\partial ^2_{ij}u(t,x)\eta^z_{\theta}(x)
\\
&&{} +\bigl[\sigma(t,x)-\sigma(t,z)\bigr]v(t,x)\eta^z_{\theta}(x)
\\
&&{}-2a^{ij}(t,z)\,\partial_iu(t,x)\,\partial_j
\eta^z_{\theta}(x) -a^{ij}(t,z)u(t,x)\,\partial^2_{ij}\eta^z_{\theta}(x)
\\
&&{}+b^i(t,x)\,\partial_iu(t,x)\eta^z_{\theta}(x)+c(t,x)u(t,x)
\eta ^z_{\theta}(x)+f(t,x)\eta^z_{\theta}(x)
\\
&= &\sum_{i=1}^7\sA_i(t,x,z,
\theta),
\end{eqnarray*}
with $\sA_i(t,x,z,\theta)$ denoting the obvious $i$th term
($i=1,2,\ldots,7$) in the three lines of sum.
Then we have $\Phi^z_{\theta}\in C^{1+\alpha}(\bR^n, L^2(\Omega)),
f^z_{\theta}\in C^{\alpha}(\bR^n, \sL^2_{\bF}(0,T))$, and
$(u^z_{\theta},v^z_{\theta})\in C^{\alpha}_{\sS^2}\cap C^{2+\alpha
}_{\sL^2}\times C^{\alpha}_{\sL^2}$. Moreover, $(u^z_{\theta
},v^z_{\theta})$ solves the following BSPDE:
\[
\cases{ -du^z_{\theta}(t,x)= \bigl[a^{ij}(t,z)\,\partial^2_{ij}u^z_{\theta
}(t,x)+f^z_{\theta}(t,x)
+\sigma^l(t,z) \bigl(v^z_{\theta}(t,x)
\bigr)_l \bigr] \,dt
\vspace*{3pt}\cr
\hspace*{62pt}{}- \bigl(v^z_{\theta}(t,x)
\bigr)_l \,dW^l_t,\qquad (t,x)\in [0,T]\times \bR^n;
\vspace*{5pt}\cr
u^z_{\theta}(T,x) =
\Phi^z_{\theta}(x),\hspace*{86pt} x\in\bR^n.}
\]
To simplify notation, define the following two types of universal constants:
\begin{eqnarray*}
C&:=&C(\lambda,\Lambda,\alpha,n,d,T),
\\
C(\cdot)&:=&C(\cdot,\lambda,\Lambda,
\alpha,n,d,T).
\end{eqnarray*}

In view of Theorem \ref{priorest1}, we have
\[
\bigl\llVert u^z_{\theta}\bigr\rrVert _{2+\alpha,\sL^2}+\bigl
\llVert v^z_{\theta}\bigr\rrVert _{\alpha,\sL
^2}\leq C \bigl(
\bigl\llVert \Phi^z_{\theta}\bigr\rrVert _{1+\alpha,L^2}+
\bigl\llVert f^z_{\theta}\bigr\rrVert _{\alpha,\sL^2} \bigr).
\]
From Lemma \ref{freezest}, we have
%e4.1 #&#
%e4.2 #&#
%
\begin{eqnarray}
\label{estuv00} \llVert u \rrVert _{2+\alpha,\sL^2}+\llVert v \rrVert
_{\alpha,\sL^2}
&\leq&  C \Bigl(\sup_z\bigl\llVert \Phi^z_{\theta}
\bigr\rrVert _{1+\alpha,L^2}+\sup_z\bigl\llVert
f^z_{\theta}\bigr\rrVert _{\alpha,\sL^2} \Bigr)
\nonumber\\[-8pt]\\[-8pt]\nonumber
&&{}   +C(\theta)
\bigl(\llVert u \rrVert _{0,\sL^2}+\llVert v \rrVert _{0,\sL^2}
\bigr).\nonumber
\end{eqnarray}
Thus, to estimate $(u,v)$, we need to estimate $\Phi^z_{\theta}$ and
$\sA_i, i=1,\ldots,7$, in terms of~$f^z_{\theta}$.
\begin{eqnarray*}
\bigl\llVert \Phi^z_{\theta}\bigr\rrVert _{1+\alpha,L^2} &=&
\bigl[\eta^z_{\theta}\Phi\bigr]_{0,L^2}+\bigl[
\eta^z_{\theta}\Phi \bigr]_{1,L^2}+\bigl[
\eta^z_{\theta}\Phi\bigr]_{1+\alpha,L^2}
\\
&\leq& C \biggl(1+\frac{1}{\theta}+\frac{1}{\theta^{1+\alpha
}} \biggr)[
\Phi]_{0,L^2}+C \biggl(1+\frac{1}{\theta^{\alpha}} \biggr)[\Phi]_{1,L^2}
\\
&&{} +\frac{C}{\theta}[\Phi]_{\alpha,L^2}+[\Phi]_{1+\alpha,L^2}
\\
&\leq& C(
\theta) \llVert \Phi \rrVert _{1+\alpha,L^2}.
\end{eqnarray*}
Denote by $[\cdot]_{m+\alpha,\sL^2,A}$ and $\llVert  \cdot  \rrVert  _{m+\alpha,\sL^2,A}$ the semi-norm and norm of functionals on subset $A\subset
\bR^n$ instead of on the whole space $\bR^n$. It is obvious that $\sA
_1(t,x,z,\theta)\equiv0$ for $x\notin B_{2\theta}(z)$. In view of
inequality~(\ref{remark21}) and the interpolation inequality in
Lemma \ref{interpolation}, we have
\begin{eqnarray*}
\bigl\llVert \sA_1(\cdot,\cdot,z,\theta)\bigr\rrVert
_{\alpha,\sL^2}&=&\bigl\llVert \sA_1(\cdot,\cdot,z,\theta)\bigr
\rrVert _{\alpha,\sL^2,B_{2\theta}(z)}
\\
&\leq& \bigl[a^{ij}(\cdot,\cdot)-a^{ij}(\cdot,z)
\bigr]_{0,\sL^{\infty
},B_{2\theta}(z)}\bigl\llVert \partial^2_{ij}u\bigr
\rrVert _{\alpha,\sL^2}
\\
&&{} +\bigl[a^{ij}(\cdot,\cdot)-a^{ij}(
\cdot,z)\bigr]_{\alpha,\sL^{\infty
}}[u]_{2,\sL^2}
\\
&&{} +\bigl[a^{ij}(\cdot,\cdot)-a^{ij}(\cdot,z)
\bigr]_{0,\sL^{\infty
},B_{2\theta}(z)}\bigl[\eta^z_{\theta}
\bigr]_{\alpha}[u]_{2,\sL^2}
\\
&\leq& \Lambda(2\theta)^{\alpha} \bigl([u]_{2+\alpha,\sL
^2}+[u]_{2,\sL^2}
\bigr)+C[u]_{2,\sL^2}
\\
&\leq& \Lambda(2\theta)^{\alpha} \bigl([u]_{2+\alpha,\sL
^2}+
\varepsilon[u]_{2+\alpha,\sL^2}+C(\varepsilon)[u]_{0,\sL
^2} \bigr)
\\
&&{} +C \bigl(\varepsilon[u]_{2+\alpha,\sL^2}+C(\varepsilon)[u]_{0,\sL
^2}
\bigr)
\\
&\leq& C \bigl(\theta^{\alpha}(1+\varepsilon)+\varepsilon
\bigr)[u]_{2+\alpha,\sL^2} +C(\varepsilon,\theta)[u]_{0,\sL^2}.
\end{eqnarray*}
Similarly, we have
\begin{eqnarray*}
\bigl\llVert \sA_2(\cdot,\cdot,z,\theta)\bigr\rrVert
_{\alpha,\sL^2}&\leq& C\theta ^{\alpha}\llVert v \rrVert
_{\alpha,\sL^2}+C[v]_{0,\sL^2},
\\
\bigl\llVert \sA_3(\cdot,\cdot,z,\theta)\bigr\rrVert
_{\alpha,\sL^2} &\leq& C[\partial_iu]_{0,\sL^2}\bigl\llvert
\partial_j\eta^z_{\theta
}\bigr\rrvert
_{\alpha} +C[\partial_iu]_{\alpha,\sL^2}\bigl[
\partial_j\eta ^z_{\theta}\bigr]_0
\\
&\leq& C \biggl(\frac{1}{\theta} + \frac{1}{\theta^{1 + \alpha
}} \biggr) \bigl[
\varepsilon[u]_{2+\alpha,\sL^2} + C(\varepsilon )[u]_{0,\sL^2} \bigr]
\\
&&{} +\frac{C}{\theta} \bigl[\varepsilon[u]_{2+\alpha,\sL^2} + C(
\varepsilon)[u]_{0,\sL^2} \bigr]
\\
&\leq& C \biggl(\frac{1}{\theta}+\frac{1}{\theta^{1+\alpha
}} \biggr)
\varepsilon[u]_{2+\alpha,\sL^2}+C(\varepsilon,\theta )[u]_{0,\sL^2},
\\
\bigl\llVert \sA_4(\cdot,\cdot,z,\theta)\bigr\rrVert
_{\alpha,\sL^2} &\leq &C[u]_{0,\sL^2}\bigl\llvert \partial^2_{ij}
\eta^z_{\theta}\bigr\rrvert _{\alpha
}+C[u]_{\alpha,\sL^2}
\bigl[\partial^2_{ij}\eta^z_{\theta}
\bigr]_0
\\
&\leq &C \biggl(\frac{1}{\theta^2}+\frac{1}{\theta^{2+\alpha
}} \biggr)[u]_{0,\sL^2}
\\
&&{}
+\frac{C}{\theta^2} \bigl(\varepsilon[u]_{2+\alpha,\sL
^2}+C_{\varepsilon}[u]_{0,\sL^2}
\bigr)
\\
&\leq &\frac{C}{\theta^2}\varepsilon[u]_{2+\alpha,\sL
^2}+C(\varepsilon,
\theta)[u]_{0,\sL^2},
\\
\bigl\llVert \sA_5(\cdot,\cdot,z,\theta)\bigr\rrVert
_{\alpha,\sL^2} &\leq& C[\partial_iu]_{0,\sL^2}\bigl\llvert
\eta^z_{\theta}\bigr\rrvert _{\alpha
}+C[
\partial_iu]_{\alpha,\sL^2}+C[\partial_iu]_{0,\sL^2}
\\
&\leq& C \biggl(1 + \frac{1}{\theta^{\alpha}} \biggr)\varepsilon [u]_{2+\alpha,\sL^2}+C(
\varepsilon,\theta)[u]_{0,\sL^2},
\\
\bigl\llVert \sA_6(\cdot,\cdot,z,\theta)\bigr\rrVert
_{\alpha,\sL^2} &\leq& C \bigl([u]_{0,\sL^2_T}\bigl\llvert
\eta^z_{\theta}\bigr\rrvert _{\alpha} +
[u]_{\alpha,\sL^2} + [u]_{0,\sL^2} \bigr)
\\
&\leq& C\varepsilon[u]_{2 + \alpha,\sL^2} + C(\varepsilon,\theta)[u]_{0,\sL^2},
\\
\bigl\llVert \sA_7(\cdot,\cdot,z,\theta)\bigr\rrVert
_{\alpha,\sL^2}&\leq& \llVert f \rrVert _{\alpha,\sL^2}+[f]_{0,\sL^2}
\bigl[\eta^z_{\theta}\bigr]_{\alpha} \leq \biggl(1+
\frac{1}{\theta^\alpha} \biggr)\llVert f \rrVert _{\alpha,\sL^2}.
\end{eqnarray*}
Choosing first $\theta$ and then $\varepsilon$ to be sufficiently
small, in view of inequality (\ref{estuv00}), we have
\begin{eqnarray*}
&& \llVert u \rrVert _{2+\alpha,\sL^2}+\llVert v \rrVert _{\alpha,\sL^2}
\\
&&\qquad \leq
\tfrac{1}{2} \bigl([u]_{2+\alpha,\sL^2} +\llVert v \rrVert _{\alpha,\sL^2}
\bigr)
\\
&&\quad\qquad{}  +C \bigl(\llVert \Phi \rrVert _{1+\alpha,L^2}+\llVert f \rrVert
_{\alpha,\sL^2} +\llVert u \rrVert _{0,\sL^2}+\llVert v \rrVert
_{0,\sL^2} \bigr).
\end{eqnarray*}
Then
%e4.3 #&#
%e4.4 #&#
%
\begin{eqnarray}
\label{estuv0} &&\llVert u \rrVert _{2+\alpha,\sL^2}+\llVert v \rrVert
_{\alpha,\sL^2}
\nonumber\\[-8pt]\\[-8pt]\nonumber
&&\qquad \leq C \bigl(\llVert \Phi \rrVert _{1+\alpha,L^2}+\llVert f \rrVert
_{\alpha,\sL^2}+\llVert u \rrVert _{0,\sL^2}+\llVert v \rrVert
_{0,\sL^2} \bigr).
\end{eqnarray}

Next, we estimate $\llVert  v \rrVert  _{0,\sL^2}$. BSPDE (\ref{bspde}) can be
written into the integral form:
%e4.5 #&#
%e4.6 #&#
%
\begin{eqnarray}
\label{bsde} u(t,x)&=&\Phi(x)\nonumber
\\
&&{} +\int_t^T
\bigl[a^{ij}(s,x)\,\partial ^2_{ij}u(s,x)+b^i(s,x)\,\partial_iu(s,x)+c(s,x)u(s,x)
\nonumber\\[-8pt]\\[-8pt]\nonumber
&&\hspace*{156pt}{}+f(s,x)+\sigma(s,x)v(s,x) \bigr] \,ds
\\
&&{}-\int_t^Tv(s,x) \,dW_s,\qquad dP\mbox{-a.s.}\nonumber
\end{eqnarray}
For any fixed $x\in\bR^n$, it is a BSDE with terminal condition $\Phi
(x)$ and generator
\[
a^{ij}(t,x)\,\partial^2_{ij}u(t,x)+b^i(t,x)\,\partial_iu(t,x)+c(t,x)U +f(t,x)+\sigma(t,x)V.
\]
We have
\begin{eqnarray*}
\hspace*{-2pt}&& E\int_0^T\bigl\llvert v(t,x)\bigr\rrvert
^2 \,dt
\\
\hspace*{-2pt}&&\qquad \leq CE \biggl[\bigl\llvert \Phi(x)\bigr\rrvert ^2+\int
_0^T\bigl\llvert a^{ij}(t,x)\,\partial
^2_{ij}u(t,x)+b^i(t,x)\,\partial_iu(t,x)
+f(t,x)\bigr\rrvert ^2 \,dt \biggr]
\\
\hspace*{-2pt}&&\qquad \leq CE \biggl[\bigl\llvert \Phi(x)\bigr\rrvert ^2+\int
_0^T \bigl[\bigl\llvert \partial
^2_{ij}u(t,x)\bigr\rrvert ^2+\bigl\llvert
\partial_iu(t,x)\bigr\rrvert ^2 +\bigl\llvert f(t,x)\bigr
\rrvert ^2 \bigr] \,dt \biggr].
\end{eqnarray*}
By the interpolation inequalities in Lemma \ref{interpolation},
\begin{eqnarray*}
\llVert v \rrVert _{0,\sL^2} &\leq &C \bigl(\llVert \Phi \rrVert
_{0,L^2}+\llVert u \rrVert _{2,\sL^2}+\llVert u \rrVert
_{1,\sL^2}+\llVert f \rrVert _{0,\sL^2} \bigr)
\\
&\leq &C\varepsilon[u]_{2+\alpha,\sL^2}+C(\varepsilon) \bigl(\llVert \Phi \rrVert
_{0,L^2}+\llVert u \rrVert _{0,\sL^2}+\llVert f \rrVert
_{0,\sL^2} \bigr).
\end{eqnarray*}
In view of (\ref{estuv0}), choosing $\varepsilon$ to be
sufficiently small, we have
%e4.7 #&#
%
\begin{equation}
\label{estuv1} \llVert u \rrVert _{2+\alpha,\sL^2}+\llVert v \rrVert
_{\alpha,\sL^2}\leq C \bigl(\llVert \Phi \rrVert _{1+\alpha,L^2}+\llVert f
\rrVert _{\alpha,\sL^2}+\llVert u \rrVert _{0,\sL^2} \bigr).
\end{equation}

We now establish a maximum principle of $u$. In BSDE (\ref{bsde}), for
any $(t,x)\in[0,T]\times\bR^n$,
%e4.8 #&#
%e4.9 #&#
%e4.10 #&#
%
\begin{eqnarray}
\label{estuv11} && E\bigl[\bigl\llvert u(t,x)\bigr\rrvert ^2\bigr]\nonumber
\\
&&\qquad \leq CE \biggl[ \bigl\llvert \Phi(x)\bigr\rrvert ^2
\nonumber
\\
&&\hspace*{54pt}{}+\int
_t^T \bigl\llvert a^{ij}(t,x)\,\partial
^2_{ij}u(t,x)+b^i(t,x)\,\partial_iu(t,x)
+f(t,x)\bigr\rrvert ^2\,dt \biggr]
\\
&&\qquad \leq CE \biggl[\int_t^T \bigl(\bigl\llvert
\partial ^2_{ij}u(t,x)\bigr\rrvert ^2+\bigl
\llvert \partial_iu(t,x)\bigr\rrvert ^2 \bigr) \,dt
\biggr]\nonumber
\\
&&\quad\qquad{} +C \bigl(\bigl\llVert \Phi(x)\bigr\rrVert ^2_{0,L^2} +
\llVert f \rrVert _{0,\sL
^2}^2 \bigr).\nonumber
\end{eqnarray}
For any $t\in[0,T]$, repeating all the preceding arguments on $[t,T]$,
we see that the estimate (\ref{estuv1}) still holds for $\llVert  \cdot  \rrVert
_{m+\alpha,\sL^2,t}$, that is,
\[
\llVert u \rrVert _{2+\alpha,\sL^2,t}\leq C \bigl(\llVert \Phi \rrVert
_{1+\alpha,L^2}+\llVert f \rrVert _{\alpha,\sL^2}+\llVert u \rrVert
_{0,\sL^2,t}\bigr).
\]
Taking supremum on both sides of (\ref{estuv11}), we have
\begin{eqnarray*}
\sup_xE\bigl[\bigl\llvert u(t,x)\bigr\rrvert
^2\bigr] &\leq&C \bigl(\llVert u \rrVert _{2+\alpha,\sL^2,t}+\bigl\llVert
\Phi(x)\bigr\rrVert ^2_{0,L^2} +\llVert f \rrVert
_{0,\sL^2}^2 \bigr)
\\
&\leq&C \bigl(\llVert u \rrVert _{0,\sL^2,t}+\bigl\llVert \Phi(x)\bigr\rrVert
^2_{1+\alpha,L^2}+\llVert f \rrVert _{\alpha,\sL^2}^2
\bigr)
\\
&\leq&C \biggl(\int_t^T \sup
_x E\bigl\llvert u(s,x)\bigr\rrvert ^2 \,ds+\bigl
\llVert \Phi(x)\bigr\rrVert ^2_{1+\alpha,L^2} +\llVert f \rrVert
_{\alpha,\sL^2}^2 \biggr).
\end{eqnarray*}
From Gronwall's inequality, we have
%e4.11 #&#
%
\begin{equation}
\label{estuv111} \qquad\llVert u \rrVert _{0,\sL^2}\leq\int_0^T
\sup_xE\bigl[\bigl\llvert u(t,x)\bigr\rrvert ^2
\bigr] \,dt\leq C \bigl(\bigl\llVert \Phi(x)\bigr\rrVert ^2_{1+\alpha,L^2}
+\llVert f \rrVert _{\alpha,\sL^2}^2 \bigr).
\end{equation}
By (\ref{estuv1}) and (\ref{estuv111}), we conclude that
%e4.12 #&#
%
\begin{equation}
\label{estuv2} \llVert u \rrVert _{2+\alpha,\sL^2}+\llVert v \rrVert
_{\alpha,\sL^2}\leq C \bigl(\llVert \Phi \rrVert _{1+\alpha,L^2}+\llVert f
\rrVert _{\alpha,\sL^2} \bigr).
\end{equation}

In a similar way, we have
%e4.13 #&#
%e4.14 #&#
%
\begin{eqnarray}
\label{estuv3} \llVert u \rrVert _{\alpha,\sS^2} &\leq& C \bigl[\llVert \Phi
\rrVert _{\alpha,L^2}+\llVert u \rrVert _{2+\alpha,\sL^2}+\llVert v \rrVert
_{\alpha,\sL^2}+\llVert f \rrVert _{\alpha,\sL^2} \bigr]
\nonumber\\[-8pt]\\[-8pt]\nonumber
&\leq& C \bigl[\llVert \Phi \rrVert _{1+\alpha,L^2}+\llVert f \rrVert
_{\alpha,\sL^2} \bigr].
\end{eqnarray}
The proof is complete.
\end{pf}

Using the method of continuation (see Gilbarg and Trudinger \cite{Gilbarg2001}, Theorem~5.2), we have from the Theorems \ref{exist1}
and \ref{priorest2} the following existence and uniqueness result
for BSPDE (\ref{bspde}).

%th4.3 #&#
%
\begin{teo}\label{exist2}
Let the Assumptions \ref{superparab2} and \ref{boundedness} be
satisfied, and
\[
(\Phi, f)\in C^{1+\alpha}\bigl(\bR^n, L^2(\Omega)
\bigr)\times C^{\alpha}\bigl(\bR ^n, \sL^2_{\bF}(0,T)
\bigr).
\]
Then BSPDE (\ref{bspde}) has a unique solution $(u,v)\in(C^\alpha
_{\sS^2}\cap C^{2+\alpha}_{\sL^2})\times C^{\alpha}_{\sL^2}$.
Moreover, there is a positive constant $C=C(\lambda,\Lambda,\alpha,n,d,T)$ such that
\[
\llVert u \rrVert _{\alpha,\sS^2}+\llVert u \rrVert _{2+\alpha,\sL^2}+\llVert v
\rrVert _{\alpha,\sL
^2}\leq C \bigl(\llVert \Phi \rrVert _{1+\alpha,L^2}+\llVert
f \rrVert _{\alpha,\sL^2}\bigr).
\]
\end{teo}

\begin{pf}
Define
\[
Lu:=a^{ij}\,\partial^2_{ij}u+b^i\,
\partial_iu+cu, \qquad Mv: =\sigma v;
\]
and for $\tau\in[0,1]$,
\[
L_{\tau}u:=(1-\tau)Lu+\tau\Delta u, \qquad M_{\tau}v:=(1-\tau
)Mv+\tau v,
\]
with $\Delta$ being the Laplacian of $\bR^n$.

Consider the following space:
\begin{eqnarray*}
\sJ^{\alpha}&:=& \biggl\{(u,v)\in\bigl(C^\alpha_{\sS^2}\cap
C^{2+\alpha
}_{\sL^2}\bigr)\times C^{\alpha}_{\sL^2}\dvtx
\forall t\in[0,T],
\\
&&\hspace*{7pt}u(t,x)=\Phi(x) + \int_t^T F(s,x)\,ds - \int
_t^T v(s,x) \,dW_s;
\\
&&\hspace*{7pt}\mbox{for some } (\Phi, F)\in C^{1+\alpha}\bigl(\bR^n,
L^2(\Omega )\bigr)\times C^{\alpha}\bigl(\bR^n,
\sL^2_{\bF}(0,T)\bigr) \biggr\},
\end{eqnarray*}
equipped with the norm of $(u,v)\in\sJ^{\alpha}$:
\[
\bigl\llVert (u,v)\bigr\rrVert _{\sJ^{\alpha}}:=\llVert u \rrVert
_{\alpha,\sS^2}+\llVert u \rrVert _{2+\alpha,\sL
^2}+\llVert v \rrVert
_{\alpha,\sL^2} +\llVert \Phi \rrVert _{1+\alpha,L^2}+\llVert F\rrVert
_{\alpha,\sL^2}.
\]
Then $\sJ^{\alpha}$ is a Banach space.

Define the mapping $\Pi_{\tau}\dvtx  \sJ^{\alpha} \rightarrow
C^{1+\alpha}(\bR^n, L^2(\Omega))\times C^{\alpha}(\bR^n, \sL
^2_{\bF}(0,T))$ as follows:
\[
\Pi_{\tau}(u,v):=(\Phi, F-L_{\tau}u-M_{\tau}v),
\qquad(u, v)\in \sJ^{\alpha}.
\]
We have
\begin{eqnarray*}
\bigl\llVert \Pi_{\tau}(u,v)\bigr\rrVert&:=& \llVert \Phi \rrVert
_{1+\alpha,L^2}+\llVert F-L_{\tau}u-M_{\tau}v\rrVert
_{\alpha,\sL
^2}
\\
&\leq& \llVert \Phi \rrVert _{1+\alpha,L^2}+\llVert F\rrVert _{\alpha,\sL^2}+
\llVert L_{\tau}u\rrVert _{\alpha,\sL^2}+\llVert M_{\tau}v
\rrVert _{\alpha,\sL^2}
\\
&\leq& C \bigl(\llVert \Phi \rrVert _{1+\alpha,L^2}+\llVert F\rrVert
_{\alpha,\sL^2}+\llVert u \rrVert _{2+\alpha,\sL^2}+\llVert v \rrVert
_{\alpha,\sL^2} \bigr)
\\
&=& C\bigl\llVert (u,v)\bigr\rrVert _{\sJ^{\alpha}}.
\end{eqnarray*}
On the other hand, for all $(t,x)\in[0,T]\times\bR^n$, we have
almost surely
\[
u(t,x)=\Phi(x)+\int_t^T \bigl[L_{\tau}u+M_{\tau}v+(F-L_{\tau
}u-M_{\tau}v)
\bigr] \,ds-\int_t^Tv(s,x) \,dW_s.
\]
Then we have from Theorem \ref{priorest2} the following estimate:
\[
\llVert u \rrVert _{\alpha,\sS^2}+\llVert u \rrVert _{2+\alpha,\sL^2}+\llVert v
\rrVert _{\alpha,\sL
^2}\leq C \bigl(\llVert \Phi \rrVert _{1+\alpha,L^2}+\llVert
F-L_{\tau}u-M_{\tau}v\rrVert _{\alpha,\sL^2}\bigr).
\]
Thus, we obtain the following inverse inequality:
\begin{eqnarray*}
\bigl\llVert (u,v)\bigr\rrVert _{\sJ^{\alpha}}
&=& \llVert u \rrVert _{\alpha,\sS^2}+\llVert u \rrVert _{2+\alpha,\sL^2}+\llVert
v \rrVert _{\alpha,\sL
^2}+\llVert \Phi \rrVert _{1+\alpha,L^2}+\llVert F\rrVert
_{\alpha,\sL^2}
\\
&\leq& \llVert u \rrVert _{\alpha,\sS^2}+\llVert u \rrVert _{2+\alpha,\sL^2}+
\llVert v \rrVert _{\alpha,\sL^2}+\llVert \Phi \rrVert _{1+\alpha,L^2}
\\
&&{} +\llVert F-L_{\tau}u-M_{\tau}v\rrVert _{\alpha,\sL^2}
+\llVert L_{\tau}u\rrVert _{\alpha,\sL^2}+\llVert
M_{\tau}v\rrVert _{\alpha,\sL^2}
\\
&\leq& C \bigl(\llVert \Phi \rrVert _{1+\alpha,L^2}+\llVert F-L_{\tau}u-M_{\tau}v
\rrVert _{\alpha,\sL^2} \bigr)
\\
&=& C\bigl\llVert \Pi_{\tau}(u,v)\bigr\rrVert.
\end{eqnarray*}
Theorem \ref{exist1} implies that $\Pi_1$ is onto. Then, in view of
the method of continuation in Gilbarg and Trudinger \cite{Gilbarg2001}, Theorem~5.2, page~75, $\Pi_{\tau}$ is also onto for all $\tau\in
[0,1)$. In particular, $\Pi_0$ is onto. The desired result follows.
\end{pf}

Similar to Proposition \ref{contintime}, we have the following H\"
{o}lder time-continuity of $u$.
%pr4.4 #&#

\begin{prop}
Let Assumptions \ref{superparab2} and \ref{boundedness} be
satisfied and $(\Phi, f)\in C^{1+\alpha}(\bR^n, L^2(\Omega))\times
C^{\alpha}(\bR^n, \sL^2_{\bF}(0,T))$. Let $(u,v)\in(C^\alpha_{\sS
^2}\cap C^{2+\alpha}_{\sL^2})\times C^{\alpha}_{\sL^2}$ solve BSPDE
(\ref{bspde}). Then, for any $\tau\in[0,T]$,
\[
\bigl\llVert u(\cdot,\cdot)-u(\cdot-\tau,\cdot)\bigr\rrVert _{\alpha,\sL^2,\tau}
\leq C \tau^{1/2} \bigl(\llVert \Phi \rrVert _{1+\alpha,L^2}+\llVert f
\rrVert _{\alpha,\sL^2} \bigr),
\]
where $C=C(\lambda,\Lambda,\alpha,T,n,d)$.
\end{prop}

At the end of the section, we discuss the consequence of the preceding
results on a deterministic PDE. Consider the deterministic functions
\begin{eqnarray*}
\Phi\dvtx \bR^n &\rightarrow&\bR,\qquad a\dvtx [0,T]\times\bR^n
\rightarrow \cS^n,
\\
b\dvtx [0,T]\times\bR^n&\rightarrow&\bR^n,\qquad\sigma\dvtx [0,T]
\times\bR ^n\rightarrow\bR^d,
\\
c,f\dvtx  [0, T]\times\bR^n&\rightarrow&\bR.
\end{eqnarray*}
As we know, a BSPDE with deterministic coefficients is in fact a
deterministic PDE. Then the second unknown variable of BSPDE (\ref
{bspde}) turns out to be 0, and BSPDE (\ref{bspde}) is in fact the
following deterministic PDE:
%e4.15 #&#
%
\begin{equation}
\label{PDE} \quad\cases{ \partial_tu(t,x)=a^{ij}(t,x)\,\partial^2_{ij}u(t,x)+b^i(t,x)\,\partial
_iu(t,x)
\vspace*{3pt}\cr
\hspace*{49pt}{} +c(t,x)u(t,x)+f(t,x),\qquad (t,x)\in[0,T)\times \bR^n;
\vspace*{5pt}\cr
u(T,x)=\Phi(x),\hspace*{118pt} x\in\bR^n,}
\end{equation}
which does not involve the coefficient $\sigma$ anymore.

Note that the classical H\"{o}lder space $C^{m+\alpha}(\bR^n)$
consists of all the deterministic elements of the H\"{o}lder space
$C^{m+\alpha}(\bR^n, L^p(\Omega))$, and the two H\"{o}lder
functional spaces $C^{m+\alpha}(\bR^n, L^p(0,T;\bR^{\iota}))$ and
$C^{m+\alpha}(\bR^n, C[0,T])$ consist of all the deterministic
elements of the two H\"{o}lder functional spaces
\[
C^{m+\alpha}\bigl(\bR^n, \sL^p_{\bF}
\bigl(0,T;\bR^{\iota}\bigr)\bigr)\quad\mbox {and}\quad C^{m+\alpha}
\bigl(\bR^n, \sS^p_{\bF}[0,T]\bigr),
\]
respectively.
Assumption \ref{boundedness} is replaced with the following one.

%as4.1 #&#
%
\begin{ass}\label{bounded1} The functions
\[
a\in C^{\alpha}\bigl(\bR^n, L^{\infty}\bigl(0,T;
\bR^{n\times n}\bigr)\bigr), \qquad b\in C^{\alpha}\bigl(
\bR^n, L^{\infty}\bigl(0,T;\bR^n\bigr)\bigr),
\]
and $c \in C^{\alpha}(\bR^n, L^{\infty}(0,T))$. There is a constant
$\Lambda>0$ such that
\[
\llVert a \rrVert _{\alpha,L^{\infty}}+\llVert b \rrVert _{\alpha,L^{\infty}}+\llVert c
\rrVert _{\alpha,L^{\infty}}\leq\Lambda.
\]
\end{ass}

In view of Theorem \ref{exist2}, we have the following existence,
uniqueness and regularity result for PDE (\ref{PDE}).
%pr4.5 #&#

\begin{prop}\label{deterPDE}
Let the Assumptions \ref{superparab2} and \ref{bounded1} be
satisfied, and
\[
(\Phi, f)\in C^{1+\alpha}\bigl(\bR^n\bigr)\times
C^{\alpha}\bigl(\bR^n, L^2(0,T)\bigr).
\]
Then PDE (\ref{PDE}) has a unique solution
\[
u\in C^\alpha\bigl(\bR^n, C[0,T]\bigr)\cap C^{2+\alpha}
\bigl(\bR^n, L^2(0,T)\bigr)
\]
such that
\[
\llVert u \rrVert _{\alpha,C}+\llVert u \rrVert _{2+\alpha,L^2}\leq C
\bigl(\llvert \Phi\rrvert _{1+\alpha}+\llVert f \rrVert _{\alpha,L^2}\bigr),
\]
where $C=C(\lambda,\Lambda,\alpha,n,d,T)$.
\end{prop}

The preceding proposition shows that the solution $u$ to PDE (\ref
{PDE}) is $(2+\alpha)$-H\"{o}lder continuous if $\Phi$ is $(1+\alpha
)$-H\"{o}lder continuous and $f$ is $\alpha$-H\"{o}lder continuous. It
seems to have a novelty as explained in the following remark.

%re4.1 #&#
%
\begin{rmk}
Mikulevicius \cite{Mikulevicius2000} studies the Cauchy problem of an
SPDE in a functional H\"{o}lder space, and includes the following a
priori estimate for PDE~(\ref{PDE}):
if $\Phi=0$, $f(t,\cdot)\in C^{\alpha}(\bR^n)$ for $t\in[0,T]$,
and $\sup_t\llvert  f(t,\cdot)\rrvert  _{\alpha}<+\infty$, then PDE~(\ref{PDE})
has a unique solution $u$ such that
\[
u(t,\cdot)\in C^{2+\alpha}\bigl(\bR^n\bigr)\qquad\forall t\in[0,T]
\quad\mbox {and}\quad\sup_t\bigl\llvert u(t, \cdot)\bigr
\rrvert _{2+\alpha}<C\sup_t\bigl\llvert f(t, \cdot )
\bigr\rrvert _{\alpha}.
\]
In contrast, in Proposition~\ref{deterPDE} we require $f\in C^{\alpha
}(\bR^n, L^2(0,T))$ and assert $u\in C^{2+\alpha}(\bR^n, L^2(0,T))$.
\end{rmk}

%s5 #&#
\section{Semi-linear BSPDEs}\label{sec5}
In this section, consider the following semi-linear BSPDE:
%e5.1 #&#
%
\begin{equation}
\label{semi-linearBSPDE} \cases{ -du(t,x)= \bigl[a^{ij}(t,x)\,\partial^2_{ij}u(t,x)+f
\bigl(t,x,\nabla u(t,x),u(t,x),v(t,x)\bigr) \bigr] \,dt
\vspace*{3pt}\cr
\hspace*{59pt}{} -v(t,x) \,dW_t,\qquad (t,x)\in[0,T)\times\bR^n,
\vspace*{5pt}\cr
u(T,x)= \Phi(x),\hspace*{76pt} x\in\bR^n.}\hspace*{-35pt}
\end{equation}
Here, $a\dvtx [0,T]\times\bR^n\rightarrow\cS^n$ satisfies both
super-parabolicity and boundedness Assumptions \ref{superparab2} and
\ref{boundedness}, $f\dvtx [0,T]\times\Omega\times\bR^n\times\bR
^n\times\bR\times\bR^d\rightarrow\bR$ is jointly measurable, and
$f(\cdot,x,q,u,v)$ is $\bF$-adapted for any $(x,q,u,v)\in\bR
^n\times\bR^n\times\bR\times\bR^d$.

We make the following Lipschitz assumption on $f$.
%as5.1 #&#

\begin{ass}\label{Lip} $f_0(\cdot,\cdot):=f(\cdot,\cdot,0,0,0)\in
C^{\alpha}(\bR^n, \sL^2_{\bF}(0,T))$,
and there is a constant $L>0$ such that
\begin{eqnarray*}
&& \bigl\llvert f(t,x,q_1,u_1,v_1)-f(t,x,q_2,u_2,v_2)
\bigr\rrvert
\\
&&\qquad \leq L\bigl(\llvert q_1-q_2\rrvert +\llvert
u_1-u_2\rrvert +\llvert v_1-v_2
\rrvert \bigr),\qquad dt\times dP\mbox{-a.e., a.s.}
\end{eqnarray*}
for any $(q_1,u_1,v_1),(q_2,u_2,v_2)\in\bR^n\times\bR\times\bR^d$
and $x\in\bR^d$.
\end{ass}

Then we have the following existence, uniqueness and regularity on
semi-linear BSPDE (\ref{semi-linearBSPDE}).
%th5.1 #&#

\begin{teo}\label{well-posed-semi-linear-BSPDE}
Let the Assumptions \ref{superparab2}, \ref{boundedness} and \ref
{Lip} be satisfied, and $\Phi\in C^{1+\alpha}(\bR^n, L^2(\Omega))$.
Then the semi-linear BSPDE (\ref{semi-linearBSPDE}) has a unique
solution $(u,v)\in(C^\alpha_{\sS^2}\cap C^{2+\alpha}_{\sL
^2})\times C^{\alpha}_{\sL^2}$. Moreover,
\[
\llVert u \rrVert _{\alpha,\sS^2}+\llVert u \rrVert _{2+\alpha,\sL^2}+\llVert v
\rrVert _{\alpha,\sL
^2}\leq C\bigl(\llVert \Phi \rrVert _{1+\alpha,L^2}+\llVert
f_0\rrVert _{\alpha,\sL^2}\bigr),
\]
where $C=C(\lambda,\Lambda,\alpha,n,d,T)$.
\end{teo}

The proof requires the following two additional preliminary lemmas.
Consider the following linear BSPDE:
%e5.2 #&#
%
\begin{equation}
\label{Beta-BSPDE} \cases{ -du(t,x)= \bigl[a^{ij}(t,x)\,\partial^2_{ij}u(t,x)-
\beta u(t,x)+f(t,x) \bigr] \,dt
\vspace*{3pt}\cr
\hspace*{57pt}{}-v(t,x) \,dW_t,\qquad (t,x)\in[0,T)\times \bR^n,
\vspace*{3pt}\cr
u(T,x)=0,\hspace*{92pt} x\in\bR^n,}
\end{equation}
where $a\dvtx [0,T]\times\bR^n\rightarrow\cS^n$ is the same as before,
and $\beta>0$ is a constant. When $a(t,x)\equiv a(t)$, define
\[
G^{\beta}_{s,t}(x):=e^{-\beta(s-t)}G_{s,t}(x),
\qquad0\leq t\leq s\leq T.
\]

%le5.2 #&#
%
\begin{lem}\label{estbeta}
For a universal constant $C=C(\lambda,\Lambda,\alpha,\gamma,n,T)$,
we have:
\begin{longlist}[(iii)]
\item[(i)]
For $\alpha\in(0,1)$ and $\gamma\in\Gamma$ such that $\llvert  \gamma\rrvert  \le2$,
%e5.3 #&#
%
\begin{equation}
\label{estibeta-1} \qquad\int_{\tau}^s\int
_{\bR^n}\bigl\llvert D^{\gamma}G^{\beta
}_{s,t}(x)
\bigr\rrvert \llvert x\rrvert ^{\alpha} \,dx \,dt\leq C\beta^{-1+ ({\llvert  \gamma\rrvert  -\alpha})/{2}},
\qquad T\geq s>\tau\geq0.
\end{equation}

\item[(ii)] For $\gamma\in\Gamma$ such that $\llvert  \gamma\rrvert  =2$ and $0\leq\tau
\leq s\leq T$,
%e5.4 #&#
%
\begin{eqnarray}
\label{estibeta-2} \int_\tau^s\biggl\llvert \int
_{\llvert  y \rrvert  \leq\eta}D^{\gamma}G^{\beta
}_{s,t}(y) \,dy
\biggr\rrvert \,dt &=&\int_\tau^s\biggl\llvert \int
_{\llvert  y \rrvert  \geq\eta}D^{\gamma}G^{\beta
}_{s,t}(y) \,dy
\biggr\rrvert \,dt
\nonumber\\[-8pt]\\[-8pt]\nonumber
&\leq& C\beta^{-1}\qquad\forall\eta>0.
\end{eqnarray}

\item[(iii)] For $\gamma\in\Gamma$ such that $\llvert  \gamma\rrvert  =2$,
%e5.5 #&#
%
\begin{equation}
\label{estibeta-3} \int_{\llvert  y \rrvert  \le\eta}\sup_{\tau\leq s}\int
_\tau^s\bigl\llvert D^{\gamma
}G^{\beta}_{s,t}(y)
\bigr\rrvert \llvert y \rrvert ^{\alpha} \,dt \,dy\leq C\beta^{-1}\eta
^{\alpha}\qquad\forall\eta>0.
\end{equation}

\item[(iv)] For any $x,\bar{x}\in\bR^n$ and $\gamma\in\Gamma$ such that
$\llvert  \gamma\rrvert  =2$,
%e5.6 #&#
%e5.7 #&#
%
\begin{eqnarray}
\label{estibeta-4} && \int_{\llvert  y-x \rrvert  >\eta}\sup_{\tau\leq s}\int
_\tau^s \bigl\llvert D^{\gamma}G^{\beta}_{s,t}(x-y)-D^{\gamma}G^{\beta
}_{s,t}(
\bar{x}-y)\bigr\rrvert \llvert \bar{x}-y\rrvert ^{\alpha} \,dt \,dy
\nonumber\\[-8pt]\\[-8pt]\nonumber
&&\qquad \leq C\beta^{-1} \llvert x-\bar{x}\rrvert ^{\alpha}
\qquad\forall\eta>0.
\end{eqnarray}
\end{longlist}
\end{lem}

%le5.3 #&#
%
\begin{lem}\label{smallest}
Let $f\in C^{\alpha}(\bR^n, \sL^2_{\bF}(0,T))$. If $(u,v)\in
(C^\alpha_{\sS^2}\cap C^{2+\alpha}_{\sL^2})\times C^{\alpha}_{\sL
^2}$ is the solution of BSPDE (\ref{Beta-BSPDE}), then
\[
\llVert u \rrVert _{\alpha,\sS^2}+\llVert u \rrVert _{2+\alpha,\sL^2}+\llVert v
\rrVert _{\alpha,\sL
^2}\leq C(\beta)\llVert f \rrVert _{\alpha,\sL^2},
\]
where $C(\beta):=C(\beta,\lambda,\Lambda,\alpha,n,d,T)>0$, and
converges to zero as $\beta\to\infty$.
\end{lem}

\begin{pf}
\textit{Step}~1 [$a(t,x)\equiv a(t)$].
Proceeding similarly as in the proof of Lemma \ref{repreu} and the
Theorems \ref{representationuv} and \ref{exist1}, we have that the
pair $(u,v)$ defined for each $x\in\bR^d$ by
\[
u(t,x):=\int_t^T  \int_{\bR^n}G^{\beta}_{s,t}(x-y)
Y(t;s,y) \,dy \,ds\qquad\forall t\in[0,T], dP\mbox{-a.s.},
\]
and
\begin{eqnarray}
v_l(t,x)&:=&\int_t^T  \int
_{\bR^n}G^{\beta}_{r,t}(x-y)
g_l(s;r,y) \,dy \,dr,\nonumber
\\
\eqntext{dt\times dP\mbox{-a.e., a.s., }l=1, \ldots,d, }
\end{eqnarray}
is the unique solution to the linear BSPDE (\ref{Beta-BSPDE}) with
\[
Y(t;\tau,x):=f(\tau,x)-\int_t^\tau
g_l(r;\tau,x) \,dW^l_r\qquad \forall t\leq
\tau.
\]
In view of the estimates of Lemma \ref{estbeta}, proceeding similarly
as in the proof of the Lemmas \ref{estvarphiY} and \ref{estpsig}
and Theorem \ref{priorest1}, we have
\[
\llVert u \rrVert _{\alpha,\sS^2}+\llVert u \rrVert _{2+\alpha,\sL^2}+\llVert v
\rrVert _{\alpha,\sL
^2}\leq C(\beta)\llVert f \rrVert _{\alpha,\sL^2},
\]
where $C(\beta):=C(\beta,\lambda,\Lambda,\alpha,n,d,T)>0$ is
sufficiently small for sufficiently large $\beta$.

\textit{Step}~2 [$(a^{ij})_{n\times n}$ depends on $x$]. Using the
freezing coefficients method as in Theorem \ref{priorest2}, we have
the desired result.
\end{pf}

\begin{pf*}{Proof of Theorem \ref{well-posed-semi-linear-BSPDE}}
For any $(U_1,V_1)\in(C^\alpha_{\sS^2}\cap C^{2+\alpha}_{\sL
^2})\times C^{\alpha}_{\sL^2}$,
 $f(\cdot,\cdot,\nabla U_1(\cdot,\cdot),U_1(\cdot,\cdot),V_1(\cdot,\cdot))\in C^{\alpha}(\bR^n,
\sL^2_{\bF}(0,T))$ because of Assumption~\ref{Lip} for $f$. In view
of Theorem \ref{exist2},
%e5.8 #&#
%
\begin{equation}
\label{semi-1} \cases{ -du_1(t,x) = \bigl[a^{ij}(t,x)\,\partial^2_{ij}u_1(t,x)
\vspace*{3pt}\cr
\hspace*{64pt}{} +f\bigl(t,x,\nabla
U_1(t,x),U_1(t,x),V_1(t,x)\bigr) \bigr] \,dt
\vspace*{3pt}\cr
\hspace*{62pt}{}-v_1(t,x) \,dW_t, \qquad(t,x)\in[0,T)\times
\bR^n,
\vspace*{5pt}\cr
u_1(T,x)=\Phi(x),\hspace*{78pt} x\in\bR^n}
\end{equation}
has a unique solution $(u_1,v_1)\in(C^\alpha_{\sS^2}\cap C^{2+\alpha
}_{\sL^2})\times C^{\alpha}_{\sL^2}$.
For any $(U_2,V_2)\in(C^\alpha_{\sS^2}\cap C^{2+\alpha}_{\sL
^2})\times C^{\alpha}_{\sL^2}$, denote\vspace*{2pt} $(u_2,v_2)\in(C^\alpha_{\sS
^2}\cap C^{2+\alpha}_{\sL^2})\times C^{\alpha}_{\sL^2}$ as the
solution of equation~(\ref{semi-1}) with $(U_1,V_1)$ replaced by
$(U_2,V_2)$. Define
\begin{eqnarray*}
\bar{u}(t,x)&:=& u_1(t,x)-u_2(t,x),\qquad\bar
{U}(t,x):=U_1(t,x)-U_2(t,x),
\\
\bar{v}(t,x)&:=& v_1(t,x)-v_2(t,x), \qquad
\bar{V}(t,x):=V_1(t,x)-V_2(t,x)
\end{eqnarray*}
and
\begin{eqnarray*}
\bar{f}(t,x)&:=& f\bigl(t,x,\nabla U_1(t,x),U_1(t,x),V_1(t,x)
\bigr)
\\
&&{}-f\bigl(t,x,\nabla U_2(t,x),U_2(t,x),V_2(t,x)
\bigr).
\end{eqnarray*}
Then we have
%e5.9 #&#
%
\begin{equation}
\label{semi-2}
\qquad\cases{ -d \bigl[e^{\beta t}\bar{u}(t,x) \bigr]=
\bigl(a^{ij}(t,x)\,\partial ^2_{ij}
\bigl[e^{\beta t}\bar{u}(t,x) \bigr]
\vspace*{3pt}\cr
\hspace*{81pt}{}-\beta e^{\beta t}\bar
{u}(t,x)+e^{\beta t}\bar{f}(t,x) \bigr) \,dt
\vspace*{3pt}\cr
\hspace*{77pt}{} -e^{\beta t}\bar{v}(t,x) (t,x) \,dW_t,\qquad (t,x)\in[0,T)\times\bR ^n,
\vspace*{5pt}\cr
e^{\beta T}\bar{u}(T,x)=0,\hspace*{132pt} x\in\bR^n.}
\end{equation}
In view of Lemma \ref{smallest}, we have
\begin{eqnarray*}
&& \bigl\llVert e^{\beta\cdot}\bar{u}\bigr\rrVert _{\alpha,\sS^2}+\bigl\llVert
e^{\beta\cdot}\bar {u}\bigr\rrVert _{2+\alpha,\sL^2} +\bigl\llVert
e^{\beta\cdot}\bar{v}\bigr\rrVert _{\alpha,\sL^2}
\\
&&\qquad \leq C(\beta)\bigl\llVert e^{\beta\cdot}\bar{f}\bigr\rrVert _{\alpha,\sL^2}
\\
&&\qquad \leq C(\beta)L \bigl[\bigl\llVert e^{\beta\cdot}\bar{U}\bigr\rrVert
_{\alpha,\sS^2}+\bigl\llVert e^{\beta\cdot}\bar{U}\bigr\rrVert
_{2+\alpha,\sL^2} +\bigl\llVert e^{\beta\cdot}\bar{V}\bigr\rrVert
_{\alpha,\sL^2} \bigr],
\end{eqnarray*}
with $C(\beta)L <1$ for a sufficiently large $\beta$. Since the
weighted norm\break $\llVert  e^{\beta\cdot}u\rrVert  _{\alpha,\sS^2}+\llVert  e^{\beta
\cdot}u\rrVert  _{2+\alpha,\sL^2}+\llVert  e^{\beta\cdot}v\rrVert  _{\alpha,\sL^2}$
is equivalent to the original one $\llVert  u \rrVert  _{\alpha,\sS^2}+ \llVert  u \rrVert  _{2+\alpha,\sL^2}+\llVert  v \rrVert  _{\alpha,\sL^2}$ in $(C^\alpha_{\sS
^2}\cap C^{2+\alpha}_{\sL^2})\times C^{\alpha}_{\sL^2}$, the
semi-linear BSPDE (\ref{semi-linearBSPDE}) has a unique solution
$(u,v)\in(C^\alpha_{\sS^2}\cap C^{2+\alpha}_{\sL^2})\times
C^{\alpha}_{\sL^2}$. The desired estimate is proved in a similar way.
\end{pf*}

\section*{Acknowledgments}
The second author thanks Xicheng Zhang and
Jinniao Qiu for their valuable comments. Both authors would like to
thank both anonymous reviewers and the Associate Editor for their
careful reading and very helpful comments and suggestions.

%\begin{appendix}
%\section{}
%\end{appendix}

%\section*{Acknowledgments}

%\begin{supplement}[id=suppA]
%\slink[doi]{10.1214/00-AOPXXXXSUPP} %[doi,text={...}] - jei reikia
%suskaldyti doi
%\sfilename{aopXXXX\_supp.pdf}
%\end{supplement}

% imsref loaded by linak, 2014-11-12 10:40:05
%

\printaddresses
\end{document}